\documentclass{ws-ijcm}
\usepackage{subfigure}
\usepackage{epsfig}

\begin{document}

%
\catchline{}{}{}{}{}
%

\markboth{Xianru Chen}
{Stable approximation of functions via Jacobi frames}

\title{STABLE APPROXIMATION OF FUNCTIONS FROM EQUISPACED SAMPLES VIA JACOBI FRAMES}

\author{Xianru Chen \footnote{Corresponding author.}}

\address{School of Mathematics and Statistics \\
Huazhong University of Science and Technology \\
Wuhan 430074, People's Republic of China  \\
\email{chenxianru@hust.edu.cn }
}


\maketitle


\begin{abstract}
In this paper, we study the Jacobi frame approximation with equispaced samples and derive an error estimation. We observe numerically that the approximation accuracy gradually decreases as the extended domain parameter $\gamma$ increases in the uniform norm, especially for differentiable functions. In addition, we show that when the indexes of Jacobi polynomials $\alpha$ and $\beta$ are larger (for example $\max\{\alpha,\beta\} > 10$), it leads to a divergence behavior on the frame approximation error decay.
\end{abstract}


\section{Introduction}
It is well-known that polynomial interpolation of functions at $m+1$ equispaced nodes will lead to the Runge's phenomenon, and many methods have been proposed to overcome the Runge's phenomenon; see  \cite{Adcock2016AMP,2009Exponentially} and reference therein. However, all these methods can not circumvent the conclusion of the impossibility theorem which was proved in \cite{Trefethen2011}, that is, any approximation procedure that achieves exponential convergence at a geometric rate must also be exponentially ill-conditioned at a geometric rate. Furthermore, as shown in \cite{Trefethen2011}, the best possible rate of convergence of a stable method is root-exponential in $m$.

Recently, Adcock, Huybrechs and Shadrin proposed an approach termed polynomial frame approximation \cite{multivariate2020,BENADCOCK}. For some fixed $\gamma>1$, polynomial frame approximation uses orthogonal polynomials on an extended interval $[-\gamma,\gamma]$ to construct an approximation to a function over $[-1,1]$. This method leads to an ill-conditioned least-squares problem, while Adcock and Shadrin showed that this problem can be computed accurately via a regularized singular value decomposition (SVD) at a set of $m+1$ linear oversampling equispaced nodes on $[-1,1]$, and proved that the regularized frame approximation operator is well-conditioned \cite{BENADCOCK}. Further, the two authors also showed that the exponential decay of the polynomial frame approximation error down to a finite user-determined tolerance $\varepsilon>0$ is indeed possible for functions that are analytic in a sufficiently large region. In other words, Adcock and Shadrin asserted the possibility of fast and stable approximation of analytic functions from equispaced samples \cite{BENADCOCK}.

When studying the theoretical analysis of polynomial frame approximation, Adcock and Shadrin adopted the Legendre polynomials for convenience \cite{BENADCOCK}. To explore the generality of the polynomial frame approximation, we next consider the use of Jacobi polynomials as well as their special cases, including Chebyshev, Legendre and Gegenbauer polynomials, are widely used in many branches of scientific computing such as approximation theory, Gauss-type quadrature and spectral methods for differential and integral equations (see, e.g., \cite{canuto2006spectral,hesthaven2007spectral,shen2011spectral,szego1975orthogonal}).
Among these applications, Jacobi polynomials are particularly appealing owing to their superior properties: (i) they have excellent error properties in the approximation of a globally smooth function; (ii) quadrature rules based on their zeros or extrema are optimal in the sense of maximizing the exactness of polynomials.

In this paper, we derive the Jacobi frame approximation error bound in Sec. 3 and we focus on the numerical experiments with various extended domain parameter $\gamma$ in Sec. 4, in particular for differentiable functions. The Jacobi frame approximation accuracy will be gradually lost as $\gamma$ increases, and it can be found that the higher the smoothness of the approximated function, the more obvious the loss of approximation accuracy. Further, we also observe numerically that when the parameters of Jacobi polynomials $\alpha$, $\beta$ are larger, for example $\mu=\max\{\alpha,\beta\} > 10$, the approximation error will become worse or even divergent.

The paper is organized as follows. In Sec. 2 we state the Jacobi frame approximation, and we derive an approximation result in Sec. 3. We then present a large number of numerical experiments of analytic functions and differentiable functions in Sec. 4.

\section{Preliminaries}

\subsection{Notations}
Let $\mathbb{P}_n$ denotes the space of polynomials of degree at most $n$ and $C(I)$ denotes the space of continuous functions on interval $I$. We define the uniform norm over $I$ as
$$\| g \|_{I,\infty} = \sup_{x \in I}|g(x)|, \quad g \in C(I).$$
For weight function
\begin{equation}\label{jacobiweight}
w^{(\alpha,\beta)}(x) = (1-x)^\alpha(1+x)^\beta, \quad \alpha,\beta> -1,
\end{equation}
we let
\begin{equation}
\langle f,g\rangle_{I,w^{(\alpha,\beta)}} = \int_I f(x)\overline{g(x)} w^{(\alpha,\beta)}(x)dx,
\quad f,g \in C(I),
\end{equation}
be the usual $L^2$-$w^{(\alpha,\beta)}$ inner product over $I$ and
$\|\cdot\|_{I,w^{(\alpha,\beta)}} = \sqrt{\langle \cdot,\cdot \rangle_{I,w^{(\alpha,\beta)}}}$ be the corresponding $L^2$-$w^{(\alpha,\beta)}$ norm.

Next, we define the discrete semi-norms and semi-inner products. For $m\geq1$, we take
$\{x_k\}^m_{k=0}$ as the $m+1$ equispaced points in $I=[-1,1]$ including endpoints. We let
\begin{equation}
\| g \|_{m,\infty} = \sup_{k = 0,...,m}|g(x_k)|, \quad g \in C(I),
\end{equation}
where $x_k = -1+2k/m $ is the equispaced grid and
\begin{equation}
\langle f,g\rangle_{m,2}
= \frac{2}{m+1} \sum^m_{k=0} f(x_k)\overline{g(x_k)}, \quad f,g \in C(I).
\end{equation}
We also let
$\|\cdot\|_{m,2} = \sqrt{\langle \cdot,\cdot \rangle_{m,2}}$ be the corresponding $L^2$ discrete semi-norm. Note that  $\|\cdot\|_{m,2}$ is an inner product on  $\mathbb{P}_n$ for any $m\geq n$. Observe that
\begin{equation}\label{norminequality}
\sqrt{2/(m+1)} \|f\|_{m,\infty}  \leq \|f\|_{m,2}
 \leq \sqrt{2}  \|f\|_{m,\infty} \leq \sqrt{2}  \|f\|_{[-1,1],\infty}, \quad f \in C(I).
\end{equation}

In this paper, we consider following families of mappings
$$\mathcal{Q}_m(f): C([-1,1])\rightarrow C([-1,1]), \quad f \in C([-1,1]),$$
where $\mathcal{Q}_m$ depends only on the values $\{f(x_i)\}^m_{i=0}$ of $f$ on the equispaced grids $\{x_i\}$ for each $m\geq 1$. Then in terms of the continuous and discrete uniform norms, we define the condition number of $\mathcal{Q}_m$ as
\begin{equation}\label{conditionnumber}
\kappa(\mathcal{Q}_m) = \sup_{f \in C([-1,1])} \lim_{\delta\rightarrow 0^+}
\sup_{\substack{q \in C([-1,1]) \\ 0\leq \|q\|_{m,\infty}\leq\delta}}
\frac{\|\mathcal{Q}_m(f+q)-\mathcal{Q}_m(f)\|_{[-1,1],\infty}}{\|q\|_{m,\infty}}.
\end{equation}
Finally, given a compact set $E \in \mathbb{C}$, we write $B(E)$ for the set of functions that are continuous on $E$ and analytic in its interior. We also define $\|f\|_{E,\infty}=\sup_{z\in E}|f(z)|$.

\subsection{Jacobi frame approximation}
We now describe polynomial frame approximation. Let $i\geq0$ be an integer and let $P_i^{(\alpha,\beta)}(x)$ denote the Jacobi polynomial of degree $i$ which is normalized by
$$ P_i^{(\alpha,\beta)}(1)=\binom{i+\alpha}{i}.$$
The sequence of Jacobi polynomials $\{P_i^{(\alpha,\beta)}(x)\}_{i=0}^\infty$ forms a system of polynomials orthogonal over the interval $[-1,1]$ with respect to weight function $w^{(\alpha,\beta)}(x)$ in \eqref{jacobiweight}, and
\begin{equation}\label{jacobiortho}
\int_{-1}^{1}  P_i^{(\alpha,\beta)}(x) P_j^{(\alpha,\beta)}(x)
w^{(\alpha,\beta)}(x)\mathrm{d}x = h_i^{(\alpha,\beta)}
\delta_{ij},
\end{equation}
where $\delta_{ij}$ is the Kronecker delta and
\begin{align}\label{asymorthoconst}
h_i^{(\alpha,\beta)} = \frac{2^{\alpha+\beta+1} \Gamma(i+\alpha+1)
\Gamma(i+\beta+1)}{(2i+\alpha+\beta+1) i! \Gamma(i+\alpha+\beta+1)},
\end{align}
and $h_i^{(\alpha,\beta)} \sim \mathcal{O}(i^{-1})$ as $ i\rightarrow \infty$.
Moreover, let $\mu= \max\{\alpha,\beta\}$, we know that \cite[Th~7.32.1]{szego1975orthogonal}
\begin{align}\label{maximumJacobi}
\|P_i^{(\alpha,\beta)}\|_{[-1,1],\infty} \sim
\left\{
\begin{aligned}
& O(n^\mu), \quad\quad \mu \geq-1/2,   \\
& O(n^{-1/2}), \quad\quad \mu< -1/2.
\end{aligned}
\right.
\end{align}

Note that an orthonormal basis on $[-\gamma,\gamma]$ fails to constitute a basis when restricted to the smaller interval $[-1,1]$, it forms the so-called polynomial frame, \cite{FCM2014,SIAMRev2019,Christensen2016}. Hence we define the Jacobi frames
 $\{\varphi^{(\alpha,\beta)}_i (x)\}^\infty_{i=0}$ as
\begin{align}\label{jacobiframe}
\varphi^{(\alpha,\beta)}_i (x) =
P_i^{(\alpha,\beta)}(x/\gamma) / \sqrt{\gamma h_i^{(\alpha,\beta)}}, \quad \gamma>1, \quad x\in [-1,1].
\end{align}
For a function $f\in C([-1,1])$, the aim of Jacobi frame approximation is to compute a polynomial approximation to $f$ of the form
$$f \approx \mathcal{J}^{\alpha,\beta}_n(f) := \sum^n_{i=0} a_i \varphi^{(\alpha,\beta)}_i (x) \in \mathbb{P}_n$$
for suitable coefficients $\{a_i\}^n_{i=0}$. The coefficients satisfy
\begin{align}\label{coefficients}
\textbf{a} = \{a_i\}^n_{i=0} \in  \mathop{\arg\min}_{\forall \mathbf{a} \in \mathbf{C}^{n+1}}
 \|A\mathbf{a} -\mathbf{b}\|_2,
\end{align}
where
$$ A \in \mathbb{C}^{(m+1)\times (n+1)}, \quad
A_{k,j} =\sqrt{2/(m+1)}\varphi^{(\alpha,\beta)}_j (x_k) , \quad 0 \leq j \leq n, \quad 0\leq k \leq m,$$
$$ \textbf{b} \in \mathbb{C}^{m+1}, \quad \textbf{b}_k = \sqrt{2/(m+1)} f(x_k), \quad 0\leq k \leq m,$$
and we let $m=\eta n$, $\eta>1$. Note that the normalization factor $\sqrt{2/(m+1)}$ is included for convenience. In fact, matrix $A$ has exponentially decaying singular values, and the least-squares problem \eqref{coefficients} is ill-conditioned for large $n$, even when $m \gg n$, due to the use of a frame rather than a basis \cite{SIAMRev2019}; see Figure~\ref{one}.

\begin{figure}[th]
\centering
{\psfig{file=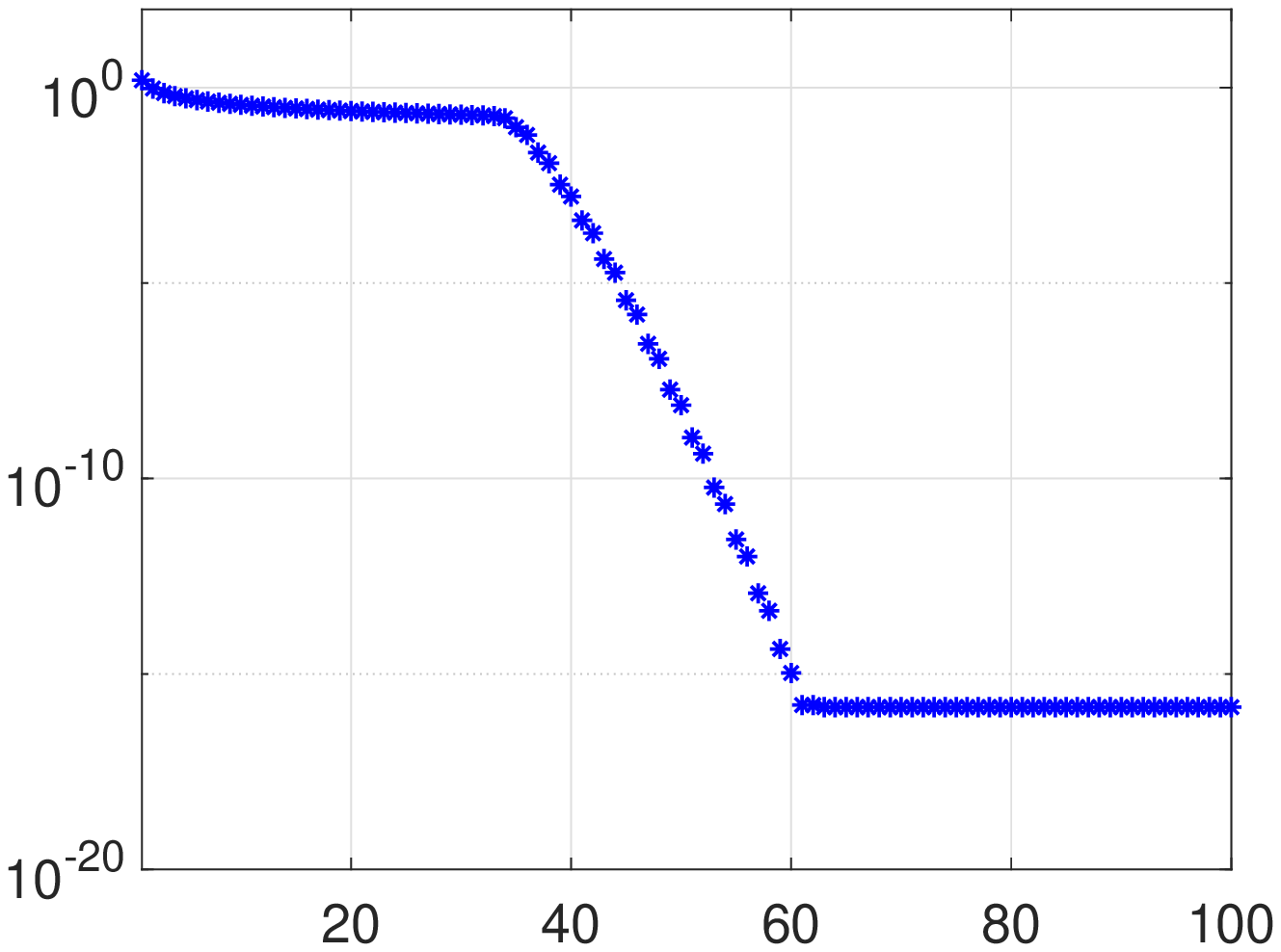,width=5cm}}
\quad
{\psfig{file=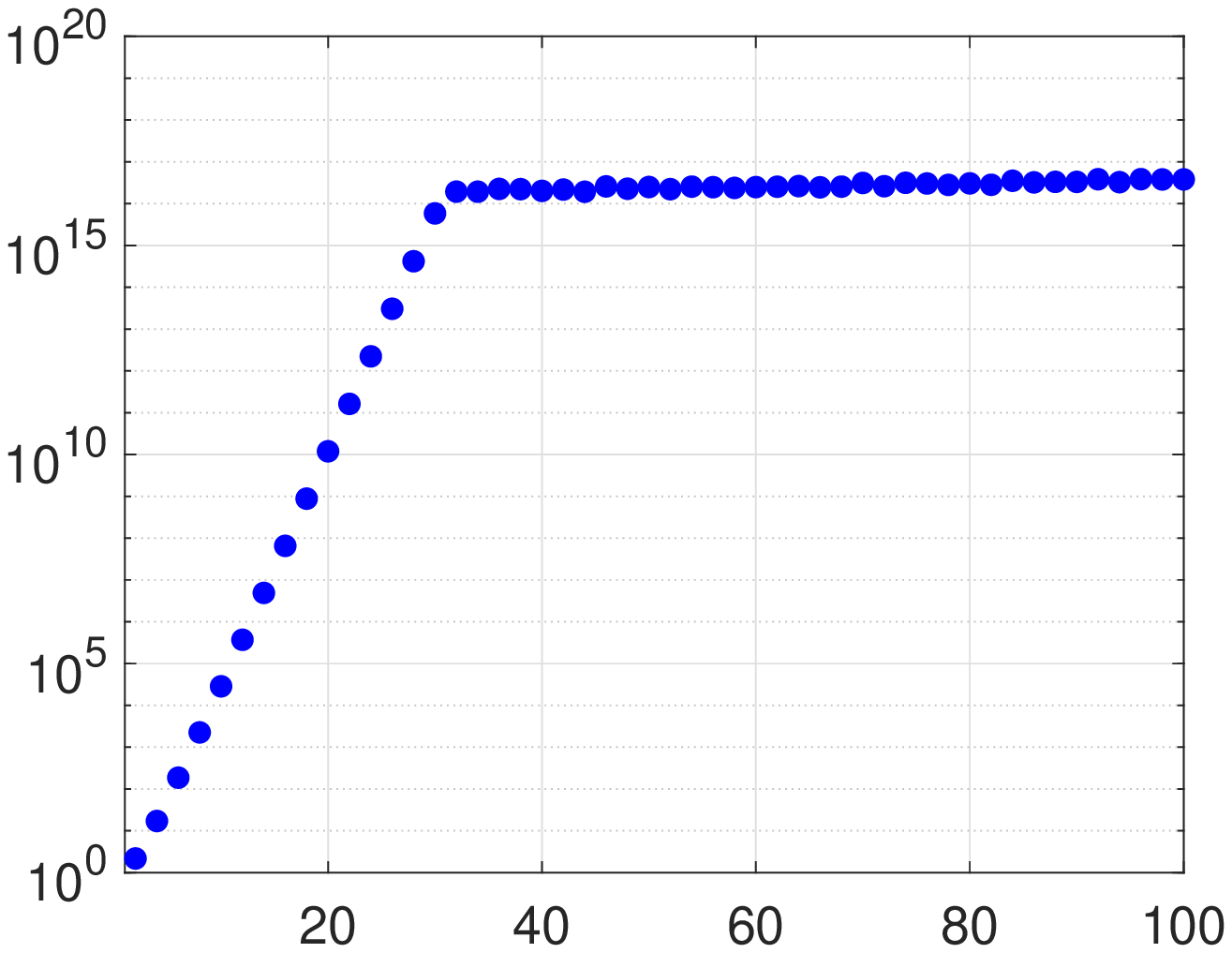,width=5cm}}
\vspace*{8pt}
\caption{The singular value profile (left) and condition number (right) of matrix $A$ versus $n$. Here, we take $\gamma=2$, $\eta=4$ and $\alpha=1/3$, $\beta=1/2$. \label{one}}
\end{figure}

In \cite{FCM2014,SIAMRev2019}, the authors proposed a truncated SVD solver with a tolerance $\varepsilon>0$ to regularize such ill-conditioned problems. In practice the user-controlled parameter $\varepsilon$ can be chosen close to machine epsilon. Despite many of the singular values which below $\varepsilon$ having been discarded, the corresponding regularized frame approximation operator can still approximate $f$ to high accuracy. Therefore, it is desired to solve \eqref{coefficients} via the truncated SVD. Let $A = U\Sigma V ^\ast$ be the SVD of matrix $A$, we define $\Sigma^\varepsilon$ as the $\varepsilon $-regularized version of $\Sigma$ and let $\Sigma^{\varepsilon,\dagger}$ be its pseudoinverse, i.e.,
\begin{align}\nonumber
(\Sigma^\varepsilon)_{ii} =
\left\{
\begin{aligned}
& \sigma_i, \quad \sigma_i > \varepsilon,   \\
& 0, \quad \sigma_i \leq \varepsilon,
\end{aligned}
\right.
\end{align}
Then we define the regularized approximation of \eqref{coefficients} as
\begin{align}
\textbf{a}^\varepsilon = V \Sigma^{\varepsilon,\dagger} U^\ast \textbf{b},
\end{align}
and the corresponding regularized Jacobi frame approximation to $f$ as
\begin{align}
\mathcal{J}^{\varepsilon,\gamma}_{m,n,\alpha,\beta}(f) :=
\sum^n_{i=0} a^\varepsilon_i \varphi^{(\alpha,\beta)}_i (x).
\end{align}
We denote the overall Jacobi frame approximation procedure as the mapping
\begin{align}\label{jacobiframeoparator}
\mathcal{Q}^{\varepsilon,\gamma}_{m,n,\alpha,\beta}: C([-1,1]) \rightarrow C([-1,1]), \quad
f\rightarrow  \mathcal{J}^{\varepsilon,\gamma}_{m,n,\alpha,\beta}(f).
\end{align}
In particular, the operator $\mathcal{Q}^{\varepsilon,\gamma}_{m,n,\alpha,\beta}$ is a standard least-squares approximation when $\varepsilon=0$ and $\gamma=1$.

\subsection{Reformulation in terms of singular vector}
In fact, another representation of the operator $\mathcal{Q}^{\varepsilon,\gamma}_{m,n,\alpha,\beta}$ can be given by means of the left or right singular vectors of the matrix $A$. Without loss of generality, here we choose the right singular vector $v_0,...,v_n \in \mathbb{C}_n$ of matrix A. Then, we can define a polynomial in $\mathbb{P}_n$ to each singular vector
\begin{align}
\zeta_i = \sum^n_{j=0} (v_i)_j  \varphi^{(\alpha,\beta)}_j(x) \in \mathbb{P}_n, \quad i= 0,...,n.
\end{align}
From the definition of $\varphi^{(\alpha,\beta)}_i (x)$, we know that $\{\varphi^{(\alpha,\beta)}_i (x)\}^\infty_{i=0}$ are the orthonormal Jacobi polynomials with weight function
$$\omega^{(\alpha,\beta)}_\gamma(x) = \omega^{(\alpha,\beta)}(x/\gamma)$$ on interval $[-\gamma,\gamma]$.
And since the $v_i$ are orthonormal vectors, the functions $\zeta_i$ are orthonormal on $[-\gamma, \gamma]$ with weight function $w_\gamma^{(\alpha,\beta)}(x)$, i.e.,
\begin{equation}\nonumber
\langle \zeta_i, \zeta_j\rangle_{[-\gamma,\gamma],w_\gamma^{(\alpha,\beta)}}
 = \delta_{ij}, \quad i,j = 0,...,n.
\end{equation}
While the functions $\zeta_i$ are also orthogonal with respect to the discrete inner product
$\langle \cdot,\cdot\rangle_{m,2}$, i.e.,
\begin{equation}\nonumber
\langle \zeta_i,\zeta_j \rangle_{m,2} = v'_j A' A v_i = \sigma^2_j \delta_{ji}, \quad i,j = 0,...,n.
\end{equation}

With this in hand, we define the subspace
\begin{equation}
\mathbb{P}^{\varepsilon,\gamma}_{m,n} = span \{\zeta_i: \sigma_i>\varepsilon \} \subseteq \mathbb{P}_n.
\end{equation}
On the one hand, this space coincides with $\mathbb{P}_n$ whenever $m\geq n$ and
$\sigma_{\min} = \sigma_n > \varepsilon$, and $\mathbb{P}^{0,\gamma}_{m,n} = \mathbb{P}_n$ for $m\geq n$ especially. On the other hand, we get
$\mathbb{P}^{\varepsilon,\gamma}_{m,n} \subseteq P^{\varepsilon,\gamma}_{m,n}$ when $\varepsilon>0$,
where
\begin{equation}
P^{\varepsilon,\gamma}_{m,n} = \{  p\in \mathbb{P}_n: \|p\|_{[-\gamma,\gamma],w_\gamma^{(\alpha,\beta)}}
\leq \|p\|_{m,2} /\varepsilon  \}\subseteq \mathbb{P}_n.
\end{equation}
Then the Jacobi frame approximation $\mathcal{Q}^{\varepsilon,\gamma}_{m,n,\alpha,\beta}$ of $f\in C([-1,1])$ belongs to the space $\mathbb{P}^{\varepsilon,\gamma}_{m,n}$. In fact, it is the orthogonal projection onto this space with respect to the discrete inner product $\langle\cdot,\cdot\rangle_{m,2}$. Therefore, by orthogonality, we write
\begin{equation}\label{Reformulation}
\mathcal{Q}^{\varepsilon,\gamma}_{m,n,\alpha,\beta}(f) = \sum_{\sigma_i>\varepsilon}
\frac{\langle f,\zeta_i \rangle_{m,2}}{\sigma^2_i} \zeta_i, \quad f\in C([-1,1]).
\end{equation}
The operator $\mathcal{Q}^{\varepsilon,\gamma}_{m,n,\alpha,\beta}$ is linear and its range is space $\mathbb{P}^{\varepsilon,\gamma}_{m,n}$.

\subsection{Two useful inequalities}
Consider the previous statements, here we give two inequalities between the uniform norm and the $L^2$-$w_\gamma^{(\alpha,\beta)}$ norm on the interval $[-\gamma,\gamma]$, which would play important roles on the error estimation in Sec. 3.2.

\begin{lemma}
Let $p \in \mathbb{P}_n$. Then
\begin{equation}\label{normconstant}
\|p\|_{[-\gamma,\gamma],\infty} \leq  c_{n,\gamma,\alpha,\beta} \|p\|_{[-\gamma,\gamma],w_\gamma^{(\alpha,\beta)}},
\end{equation}
where the constant $c_{n,\gamma,\alpha,\beta}$ has the following asymptotic property
as $n\rightarrow\infty $
\begin{align}\label{imimimconstant}
c_{n,\gamma,\alpha,\beta}
= \max_{x\in [-\gamma,\gamma]} \left(\sum^n_{i=0} |\varphi^{(\alpha,\beta)}_i (x)|^2 \right)^{1/2}
=
\left\{
\begin{aligned}
& \mathcal{O}(n^{\mu+1}/\sqrt{\gamma}), \quad\mu > -1/2,   \\
& \mathcal{O}(n^{1/2}/\sqrt{\gamma}),  \quad\mu \leq -1/2.
\end{aligned}
\right.
\end{align}
\end{lemma}

\begin{proof}
We write $ p\in \mathbb{P}_n$ as $p =\sum^n_{i=0} c_i \varphi^{(\alpha,\beta)}_i (x) $. Using the Cauchy-Schwartz inequality and recalling \eqref{jacobiframe}, we have
\begin{equation}\nonumber
\|p\|_{[-\gamma,\gamma],\infty}
 \leq \left(\sum^n_{i=0} |c_i|^2\right)^{1/2}
       \left(\sum^n_{i=0} |\varphi^{(\alpha,\beta)}_i (x)|^2 \right)^{1/2} \\
\leq \|p\|_{[-\gamma,\gamma],w_\gamma^{(\alpha,\beta)}} c_{n,\gamma,\alpha,\beta},
\end{equation}
where
\begin{equation}\nonumber
c_{n,\gamma,\alpha,\beta}
= \frac{1}{\sqrt{\gamma}}\left(\sum^n_{i=0} \frac{1}{h_i^{(\alpha,\beta)}}
   \left(\max_{x\in [-\gamma,\gamma]} |P_i^{(\alpha,\beta)}(x/\gamma)|\right)^2\right)^{1/2}.
\end{equation}
We can directly show the asymptotic property of $c_{n,\gamma,\alpha,\beta}$ directly by recalling \eqref{maximumJacobi} and \eqref{asymorthoconst} for the classical Jacobi polynomial
$P^{(\alpha,\beta)}_i (x)$.
\end{proof}

\begin{lemma}
Let $p \in \mathbb{P}_n$. Then
\begin{equation}\label{inverseinequa}
\|p\|_{[-\gamma,\gamma],w_\gamma^{(\alpha,\beta)}}
\leq \sqrt{\gamma h_0^{(\alpha,\beta)} } \|p\|_{[-\gamma,\gamma],\infty}.
\end{equation}
\end{lemma}

\begin{proof}
According to the definition of $L^2$ norm, we have
\begin{equation}\nonumber
\|p\|^2_{[-\gamma,\gamma],w_\gamma^{(\alpha,\beta)}}
\leq \|p\|^2_{[-\gamma,\gamma],\infty} \int^\gamma_{-\gamma}w_\gamma^{(\alpha,\beta)}(x)dx
 = \gamma h_0^{(\alpha,\beta)} \|p\|^2_{[-\gamma,\gamma],\infty},
\end{equation}
where $h_0^{(\alpha,\beta)}$ defined in \eqref{asymorthoconst} as
$$h_0^{(\alpha,\beta)} = \frac{2^{\alpha+\beta+1} \Gamma(\alpha+1)
\Gamma(\beta+1)}{(\alpha+\beta+1) \Gamma(\alpha+\beta+1)}.$$
\end{proof}

\section{Accuracy and Conditioning of Jacobi Frame Approximation}

\subsection{A rough error bound}
According to the definition of norms defined in Sec. 2 and \eqref{conditionnumber}, combined with norm inequality \eqref{norminequality} and the triangle inequality, we deduce the following results. Further, this theorem also holds for Chebyshev and Gegenbauer polynomial frame approximation. Since the specific idea is completely consistent with Lemma 3.1 in \cite{BENADCOCK}, we omit the derivation process.
\begin{theorem}
Let $\varepsilon>0$, $\gamma>1$ and $\mathcal{Q}^{\varepsilon,\gamma}_{m,n,\alpha,\beta}$ be the Jacobi frame approximation operator defined in \eqref{jacobiframeoparator}. Then for any $f\in C([-1,1])$
\begin{equation}\label{firsterror}
\begin{split}
& \quad \|f-\mathcal{Q}^{\varepsilon,\gamma}_{m,n,\alpha,\beta}(f) \|_{[-1,1],\infty} \\
&\leq (1+\sqrt{m+1}C_1) \|f-p\|_{[-1,1],\infty} + C_2 \varepsilon \|p\|_{[-\gamma,\gamma],\infty},
\quad \forall p \in \mathbb{P}_n,
\end{split}
\end{equation}
where
\begin{align}\label{firstconstant}
C_1 := C_1(m,n,\alpha,\beta,\gamma,\varepsilon) = \sup\{ \|p\|_{[-1,1],\infty}:
p \in \mathbb{P}^{\varepsilon,\gamma}_{m,n},  \|p\|_{m,\infty}\leq1\},
\end{align}
\begin{align}\label{secondconstant}
C_2 := C_2(m,n,\alpha,\beta,\gamma,\varepsilon)
= \{ \|p-\mathcal{Q}^{\varepsilon,\gamma}_{m,n,\alpha,\beta}(p)\|_{[-1,1],\infty}: p \in \mathbb{P}_n,
\|p\|_{[-\gamma,\gamma],\infty} \leq \varepsilon^{-1} \}.
\end{align}
Moreover, the condition number $\kappa(\mathcal{Q}^{\varepsilon,\gamma}_{m,n,\alpha,\beta})$ satisfies \begin{align}\label{firstcond}
C_1\leq \kappa(\mathcal{Q}^{\varepsilon,\gamma}_{m,n,\alpha,\beta}) \leq \sqrt{m+1}C_1.
\end{align}
\end{theorem}
The next work is to prove that the constants $C_1,C_2$ are bounded under some assumptions.

\subsection{Bounding the constants $C_1$ and $C_2$}
\begin{theorem}
The constants $C_1$ and $C_2$ defined in \eqref{firstconstant}, \eqref{secondconstant} satisfy
\begin{align}\label{firstbound}
C_1 \leq C\left(m,n,\gamma,\varepsilon/ (\sqrt{2} c_{n,\gamma,\alpha,\beta} )\right),
\end{align}
\begin{align}\label{secondbound}
C_2 \leq \sqrt{\gamma(m+1)h_0^{(\alpha,\beta)}/2} \cdot
C\left(m,n,\gamma,\sqrt{m+1}\varepsilon / (\sqrt{2}c_{n,\gamma,\alpha,\beta})\right),
\end{align}
where
\begin{align}\label{thirdbound}
C(m,n,\gamma,\varepsilon) := \sup\{\|p\|_{[-1,1],\infty}: p \in \mathbb{P}_n,   \|p\|_{m,\infty}\leq1, \|p\|_{[-\gamma,\gamma],\infty} \leq \varepsilon^{-1}\}.
\end{align}
\end{theorem}

\begin{proof}
We first consider constant $C_1$. Let $p\in\mathbb{P}^{\varepsilon,\gamma}_{m,n}$ with $\|p\|_{m,\infty}\leq1$. Then we can write $p = \sum_{\sigma_i>\varepsilon} c_i \zeta_i $ and, using the  the continuous and discrete orthogonality of the $\zeta_i$, we get
\begin{equation}\nonumber
\|p\|_{[-\gamma,\gamma],w_\gamma^{(\alpha,\beta)}}
\leq \|p\|_{m,2} / \varepsilon
\leq \sqrt{2} \|p\|_{m,\infty} / \varepsilon
\leq \sqrt{2}/ \varepsilon .
\end{equation}
Since $p\in\mathbb{P}^{\varepsilon,\gamma}_{m,n}\subseteq \mathbb{P}_n$, using Lemma 2.1, we have
\begin{equation}
\|p\|_{[-\gamma,\gamma],\infty}
\leq  c_{n,\gamma,\alpha,\beta} \|p\|_{[-\gamma,\gamma],w_\gamma^{(\alpha,\beta)}}
\leq  \sqrt{2} c_{n,\gamma,\alpha,\beta} / \varepsilon .
\end{equation}
We deduce that
\begin{equation}\nonumber
\begin{split}
C_1 &\leq \sup \left\{ \|p\|_{[-1,1],\infty}: p \in \mathbb{P}^{\varepsilon,\gamma}_{m,n},
       \|p\|_{m,\infty}\leq1, \|p\|_{[-\gamma,\gamma],\infty}
      \leq  \sqrt{2} c_{n,\gamma,\alpha,\beta}/ \varepsilon \right\} \\
&= C\left(m,n,\gamma,\varepsilon/ (\sqrt{2} c_{n,\gamma,\alpha,\beta} )\right).
\end{split}
\end{equation}

We then consider constant $C_2$. Let $p\in\mathbb{P}^{\varepsilon,\gamma}_{m,n}$ and $\|p\|_{[-\gamma,\gamma],\infty}\leq \varepsilon^{-1}$. Since $p \in \mathbb{P}_n$ and \eqref{Reformulation}, we may write
$$ p= \sum^n_{i=0}
\frac{\langle p,\zeta_i \rangle_{m,2}} {\sigma^2_i} \zeta_i, \quad
\mathcal{Q}^{\varepsilon,\gamma}_{m,n,\alpha,\beta}(p) = \sum_{\sigma_i>\varepsilon}
\frac{\langle p,\zeta_i \rangle_{m,2}}{\sigma^2_i} \zeta_i.$$
Using the continuous and discrete orthogonality of the $\zeta_i$ again, we get that
\begin{equation}\label{discreteortho}
\|p-\mathcal{Q}^{\varepsilon,\gamma}_{m,n,\alpha,\beta}(p)\|^2_{[-\gamma,\gamma],w_\gamma^{(\alpha,\beta)}}
 \leq \sum^n_{i=0} |\langle p,\zeta_i \rangle_{m,2}|^2 / \sigma^4_i
= \|p\|^2_{[-\gamma,\gamma],w_\gamma^{(\alpha,\beta)}}.
\end{equation}
\begin{equation}\label{zeropart}
\|p-\mathcal{Q}^{\varepsilon,\gamma}_{m,n,\alpha,\beta}(p)\|^2_{m,2}
\leq \varepsilon^2 \sum_{\sigma_i\leq\varepsilon} |\langle p,\zeta_i \rangle_{m,2}|^2 / \sigma^4_i
\leq  \varepsilon^2 \|p\|^2_{[-\gamma,\gamma],w_\gamma^{(\alpha,\beta)}}.
\end{equation}

Now observe that we can write
$$C_2= \max\{ \|q\|_{[-1,1],\infty}: q\in \mathcal{A}\},$$
$$ \mathcal{A} = \{q: q=p-\mathcal{Q}^{\varepsilon,\gamma}_{m,n,\alpha,\beta}(p),
p\in \mathbb{P}_n,  \|p\|_{[-\gamma,\gamma],\infty}\leq \varepsilon^{-1}\}.$$
Let  $q=p-\mathcal{Q}^{\varepsilon,\gamma}_{m,n,\alpha,\beta}(p) \in \mathcal{A}$. Then $q\in \mathbb{P}_n$ and due to \eqref{discreteortho}, \eqref{normconstant} and \eqref{inverseinequa},
\begin{equation}\label{firstpart}
\begin{split}
\|q\|_{[-\gamma,\gamma],\infty}
& \leq c_{n,\gamma,\alpha,\beta} \|q\|_{[-\gamma,\gamma],w_\gamma^{(\alpha,\beta)}}
\leq c_{n,\gamma,\alpha,\beta} \|p\|_{[-\gamma,\gamma],w_\gamma^{(\alpha,\beta)}}  \\
& \leq c_{n,\gamma,\alpha,\beta}  \sqrt{\gamma h_0^{(\alpha,\beta)}}  \|p\|_{[-\gamma,\gamma],\infty}
\leq c_{n,\gamma,\alpha,\beta}\sqrt{\gamma h_0^{(\alpha,\beta)}}  \varepsilon^{-1}.  \\
\end{split}
\end{equation}
By \eqref{norminequality} , \eqref{zeropart} and \eqref{inverseinequa}, and the fact that
$\|p\|_{[-\gamma,\gamma],\infty}\leq 1/\varepsilon$, we have
\begin{equation}\label{secondpart}
\begin{split}
\|q\|_{m,\infty}
&\leq \sqrt{(m+1)/2} \|p-\mathcal{Q}^{\varepsilon,\gamma}_{m,n,\alpha,\beta}(p)\|_{m,2}
\leq  \sqrt{(m+1)/2}  \varepsilon \|p\|_{[-\gamma,\gamma],w_\gamma^{(\alpha,\beta)}} \\
& \leq \sqrt{(m+1)/2}  \varepsilon \sqrt{\gamma h_0^{(\alpha,\beta)}}
\|p\|_{[-\gamma,\gamma],\infty}
= \sqrt{(m+1)/2} \sqrt{\gamma h_0^{(\alpha,\beta)}}.
\end{split}
\end{equation}
Hence, $q \in \mathcal{B}$ where
\begin{equation}\nonumber
\mathcal{B}:= \left\{ q\in \mathbb{P}_n,\|q\|_{[-\gamma,\gamma],\infty}
    \leq c_{n,\gamma,\alpha,\beta} \sqrt{\gamma h_0^{(\alpha,\beta)}}  \varepsilon^{-1} ,
\|q\|_{m,\infty} \leq \sqrt{\gamma(m+1)h_0^{(\alpha,\beta)}/2} \right\},
\end{equation}
which implies that $C_2\leq \max\{ \|q\|_{[-1,1],\infty}: q\in \mathcal{B}\}$ and after renormalizing, we get
\begin{equation}\nonumber
\begin{split}
C_2
&\leq \sqrt{\gamma(m+1)h_0^{(\alpha,\beta)}/2} \\
& \times \max \left\{ \|p\|_{[-1,1],\infty}: p \in \mathbb{P}_n, \|p\|_{m,\infty}\leq 1,
\|p\|_{[-\gamma,\gamma],\infty}
\leq  \sqrt{2}c_{n,\gamma,\alpha,\beta} / (\sqrt{m+1} \varepsilon)
   \right\} \\
&= \sqrt{\gamma(m+1)h_0^{(\alpha,\beta)}/2} \cdot
C\left(m,n,\gamma,\sqrt{m+1}\varepsilon / (\sqrt{2}c_{n,\gamma,\alpha,\beta})\right).
\end{split}
\end{equation}
This completes the proof.
\end{proof}

We now can conclude that the constants $C_1$ and $C_2$ both are bounded with the following assumptions proved in \cite{BENADCOCK}.
\begin{theorem}(\cite{BENADCOCK})
Let $0\leq \varepsilon \leq 1/e$, $\gamma\geq1$ and $n\geq\sqrt{\gamma^2-1} \log(1/\varepsilon)$, and consider the quantity $C(m,n,\gamma,\varepsilon)$ defined in \eqref{thirdbound}. Suppose that
\begin{equation}\label{oversamping}
m \geq 36n  \log(\varepsilon^{-1}) / \sqrt{\gamma^2-1}.
\end{equation}
Then $C(m,n,\gamma,\varepsilon) \leq c$ for some numerical constant $c>0$.
\end{theorem}

\subsection{Main results}
We now summarize Theorem 3.1-3.3 and prove the main result of this section.
\begin{theorem}
Let $\varepsilon>0$, $\gamma\geq1$, $c>1$, $\mu=\max\{\alpha,\beta\}$ and $m\geq n\geq1$ satisfies \eqref{oversamping} be such that
\begin{align}\label{boundconstant}
C(m,n,\gamma,\varepsilon) \leq c.
\end{align}
Then the Jacobi frame approxiamtion $\mathcal{Q}^{\varepsilon',\gamma}_{m,n,\alpha,\beta}$ with
$ \varepsilon'= \sqrt{2} c_{n,\gamma,\alpha,\beta}  \varepsilon$
satisfies
\begin{align}
\kappa(\mathcal{Q}^{\varepsilon',\gamma}_{m,n,\alpha,\beta}) \leq c \sqrt{m+1}.
\end{align}
And for any $f\in C([-1,1])$,
\begin{equation}\label{apperror}
\begin{split}
& \quad \|f-\mathcal{Q}^{\varepsilon',\gamma}_{m,n,\alpha,\beta}(f) \|_{[-1,1],\infty} \\
& \leq 2c \sqrt{m+1} \\
& \times
\left\{
\begin{aligned}
&\inf_{p \in \mathbb{P}_n}\left(\|f-p\|_{[-1,1],\infty}
+ \sqrt{h_0^{(\alpha,\beta)}} /2 \cdot
n^{\mu+1}\varepsilon \|p\|_{[-\gamma,\gamma],\infty} \right), \quad \mu > -1/2,   \\
&\inf_{p \in \mathbb{P}_n} \left( \|f-p\|_{[-1,1],\infty}
+\sqrt{h_0^{(\alpha,\beta)}} /2 \cdot
 n^{1/2}\varepsilon \|p\|_{[-\gamma,\gamma],\infty} \right), \quad \mu \leq -1/2.
\end{aligned}
\right.
\end{split}
\end{equation}
\end{theorem}

\begin{proof}
Observe that $C(m,n,\gamma,\varepsilon)$ is a decreasing function of $\varepsilon$. Hence, Theorem 3.4, Theorem 3.5 and the condition \eqref{boundconstant} imply that,
\begin{equation}\nonumber
C_1(m,n,\alpha,\beta,\gamma,\varepsilon')
\leq C\left(m,n,\gamma,\varepsilon'/ (\sqrt{2} c_{n,\gamma,\alpha,\beta})  \right)
= C(m,n,\gamma,\varepsilon) \leq c,
\end{equation}
and
\begin{equation}\nonumber
\begin{split}
C_2(m,n,\alpha,\beta,\gamma,\varepsilon')
&\leq \sqrt{\gamma(m+1)h_0^{(\alpha,\beta)}/2} \cdot
C\left(m,n,\gamma,\sqrt{m+1}\varepsilon'/ (\sqrt{2}c_{n,\gamma,\alpha,\beta}) \right) \\
&\leq \sqrt{\gamma(m+1)h_0^{(\alpha,\beta)}/2} \cdot C(m,n,\gamma,\varepsilon)
\leq c \sqrt{\gamma(m+1)h_0^{(\alpha,\beta)}/2}.
\end{split}
\end{equation}
We now apply Theorem 3.1 to get
$$\kappa(\mathcal{Q}^{\varepsilon',\gamma}_{m,n,\alpha,\beta}) \leq \sqrt{m+1}C_1 \leq c\sqrt{m+1}. $$
For the error bound, we have
\begin{equation}\nonumber
\begin{split}
&\|f-\mathcal{Q}^{\varepsilon',\gamma}_{m,n,\alpha,\beta}(f) \|_{[-1,1],\infty} \\
& \leq (1+c \sqrt{m+1}) \|f-p\|_{[-1,1],\infty}
+  c\sqrt{\gamma(m+1)h_0^{(\alpha,\beta)}/2} \cdot \varepsilon' \|p\|_{[-\gamma,\gamma],\infty}\\
& \leq 2c\sqrt{m+1}  \left( \|f-p\|_{[-1,1],\infty}
+ \sqrt{\gamma h_0^{(\alpha,\beta)}}/2 \cdot c_{n,\gamma,\alpha,\beta} \varepsilon
\|p\|_{[-\gamma,\gamma],\infty}  \right)
\end{split}
\end{equation}
where in the second step we used the definition of $\varepsilon'$. The result now follows, since
$1+c \sqrt{m+1} \leq 2c \sqrt{m+1}$. Combined \eqref{imimimconstant}, we finish the proof.
\end{proof}

As a result, the overall Jacobi frame approximation error depends on how well $f$ can be approximated by a polynomial $p \in \mathbb{P}_n$ uniformly on $[-1,1]$. Moreover, we can give a specific decay rate of the term $\|f-p\|_{[-1,1],\infty}$ in \eqref{apperror} for analytic functions and differentiable functions.
Under the premise of keeping all the assumptions of Theorem 3.4 unchanged, we give the Jacobi frame approximation error estimates for analytic functions and differentiable functions without derivations in Theorem 3.5, 3.6 and 3.7 respectively.

\begin{theorem}
Let $E_\theta$ be the Bernstein ellipse with parameter $\theta > \tau := \gamma + \sqrt{\gamma^2-1}$. Then for all $f\in B(E_\theta)$,
\begin{equation}\label{firstanalyticTh1}
\begin{split}
&\quad \|f-\mathcal{Q}^{\varepsilon',\gamma}_{m,n,\alpha,\beta}(f) \|_{[-1,1],\infty} \\
& \leq c\sqrt{m+1}\|f\|_{E_\theta,\infty} \\
&\quad \times
\left\{
\begin{aligned}
& G^A_1(\theta) (\theta^{-n}+ n^{\mu+1} \varepsilon),
\quad G^A_3(\theta,\gamma,\alpha,\beta)\leq 1, \quad \mu\geq -1/2,  \\
& G^A_1(\theta) (\theta^{-n}+ n^{1/2} \varepsilon),
\quad G^A_3(\theta,\gamma,\alpha,\beta)\leq 1, \quad \mu < -1/2,  \\
& G^A_2(\theta,\gamma,\alpha,\beta) (\theta^{-n}+ n^{\mu+1} \varepsilon),
\quad G^A_3(\theta,\gamma,\alpha,\beta) > 1, \quad \mu\geq -1/2,    \\
& G^A_2(\theta,\gamma,\alpha,\beta) (\theta^{-n}+ n^{1/2} \varepsilon),
\quad G^A_3(\theta,\gamma,\alpha,\beta) > 1, \quad \mu< -1/2.    \\
\end{aligned}
\right.
\end{split}
\end{equation}
where the constants defined as
\begin{equation}
G^A_1(\theta)=\frac{4}{\theta-1}, \quad
G^A_2(\theta,\gamma,\alpha,\beta)=
\frac{2\theta \sqrt{h_0^{(\alpha,\beta)}}}{\theta-\gamma-\sqrt{\gamma^2-1}},
\end{equation}
\begin{equation}
 G^A_3(\theta,\gamma,\alpha,\beta)
 =\frac{G^A_2(\theta,\gamma,\alpha,\beta)}{G^A_1(\theta)}
=\frac{\theta(\theta-1)\sqrt{h_0^{(\alpha,\beta)}}}{2(\theta-\gamma-\sqrt{\gamma^2-1})}.
\end{equation}
\end{theorem}

\begin{theorem}
Let $E_\theta$ be the Bernstein ellipse with parameter $1 <\theta< \tau :=\gamma + \sqrt{\gamma^2-1}$. Then for all $f\in B(E_\theta)$,
\begin{equation}\label{secondanalyticTh1}
\begin{split}
&\quad \|f-\mathcal{Q}^{\varepsilon',\gamma}_{m,n,\alpha,\beta}(f) \|_{[-1,1],\infty} \\
&\leq c\sqrt{m+1} \|f\|_{E_\theta,\infty} \\
& \times
\left\{
\begin{aligned}
&  H^A_1(\theta)
\left(\theta^{-n}+ (1+n^{\mu+1}) \varepsilon^{\frac{\log\theta}{\log\tau}} \right),
\quad H^A_3(\theta,\gamma,\alpha,\beta)\leq 1, \quad \mu\geq -1/2,  \\
&  H^A_1(\theta)
\left(\theta^{-n}+ (1+n^{1/2}) \varepsilon^{\frac{\log\theta}{\log\tau}} \right),
\quad H^A_3(\theta,\gamma,\alpha,\beta)\leq 1, \quad \mu < -1/2,  \\
&  H^A_2(\theta,\gamma,\alpha,\beta)
\left(\theta^{-n}+ (1+n^{\mu+1}) \varepsilon^{\frac{\log\theta}{\log\tau}} \right),
\quad H^A_3(\theta,\gamma,\alpha,\beta) > 1, \quad \mu\geq -1/2,    \\
& H^A_2(\theta,\gamma,\alpha,\beta)
\left(\theta^{-n}+ (1+n^{1/2}) \varepsilon^{\frac{\log\theta}{\log\tau}} \right),
\quad H^A_3(\theta,\gamma,\alpha,\beta) > 1, \quad \mu< -1/2.    \\
\end{aligned}
\right.
\end{split}
\end{equation}
where the constants defined as
\begin{equation}
H^A_1(\theta)=\frac{4}{\theta-1}, \quad
H^A_2(\theta,\gamma,\alpha,\beta)
= \frac{2(\gamma+\sqrt{\gamma^2-1})\sqrt{h_0^{(\alpha,\beta)}}}{\gamma+\sqrt{\gamma^2-1}-\theta},
\end{equation}
\begin{equation}
H^A_3(\theta,\gamma,\alpha,\beta)
 =\frac{H^A_2(\theta,\gamma,\alpha,\beta)}{H^A_1(\theta)}
=\frac{(\gamma+\sqrt{\gamma^2-1} )(\theta-1)\sqrt{h_0^{(\alpha,\beta)}}}
{2(\gamma+\sqrt{\gamma^2-1}-\theta)}.
\end{equation}
\end{theorem}

\begin{theorem}
For all $k \in \mathbb{N}$ and $f\in C^k([-1,1])$,
\begin{equation}
\begin{split}
&\quad \|f-\mathcal{Q}^{\varepsilon',\gamma}_{m,n,\alpha,\beta}(f) \|_{[-1,1],\infty} \\
& \leq 4c_{k,\gamma} c'_{k,\gamma} c \sqrt{m+1} \|f\|_{C^k([-1,1])}
\left\{
\begin{aligned}
\left(n^{-k} +\sqrt{h_0^{(\alpha,\beta)}}/(4c'_{k,\gamma}) \cdot n^\mu \varepsilon \right),
\quad \mu\geq -1/2,  \\
\left(n^{-k} +\sqrt{ h_0^{(\alpha,\beta)}}/(4c'_{k,\gamma}) \cdot n^{1/2} \varepsilon \right),
\quad \mu<-1/2.
\end{aligned}
\right.
\end{split}
\end{equation}
The constants $c_{k,\gamma} $ and $ c'_{k,\gamma}$ depend on $k,\gamma$.
\end{theorem}

Notice, we show that the results in \cite{BENADCOCK} are special case of this paper, that is, $\alpha=\beta=0$. And the error decreases only down to a constant tolerance $\varepsilon$ when the value of $\mu$ is appropriate. Once the value of $\mu$ is larger, the second term in brackets on the right side of equation \eqref{apperror} becomes more dominant and even leads to error divergence. Further, the value of $\gamma$ also influences the approximation error, as detailed in the numerical experiments in Sec. 4.

\section{Numerical Experiments}
In the following experiments, we compute the uniform error of the approximation with threshold parameter $\varepsilon$ rather than $\varepsilon'$ on a grid of 10,000 equispaced points in $[-1, 1]$.

\subsection{The influences of parameters $\varepsilon$ and $\eta$}
Let $(\alpha,\beta)=(1/3,1/2)$ to be fixed in Figure~\ref{two}-Figure~\ref{four} firstly. We show the uniform Jacobi frame approximation error versus $n$ of functions
$f_1 = 1/(1+x^2)$ for various values of $\gamma,\varepsilon,\eta$ in Figure~\ref{two}. This function is analytic in $B(E_\theta)$ with parameters $\theta_1=1+\sqrt{2}$, where $E_{\theta}$ is defined as $E_\theta=\{(z+z^{-1})/2: z\in \mathbb{C}, 1\leq |z|\leq \theta\}$.

\begin{figure}[bh]
\centering
{\includegraphics[width=5cm]{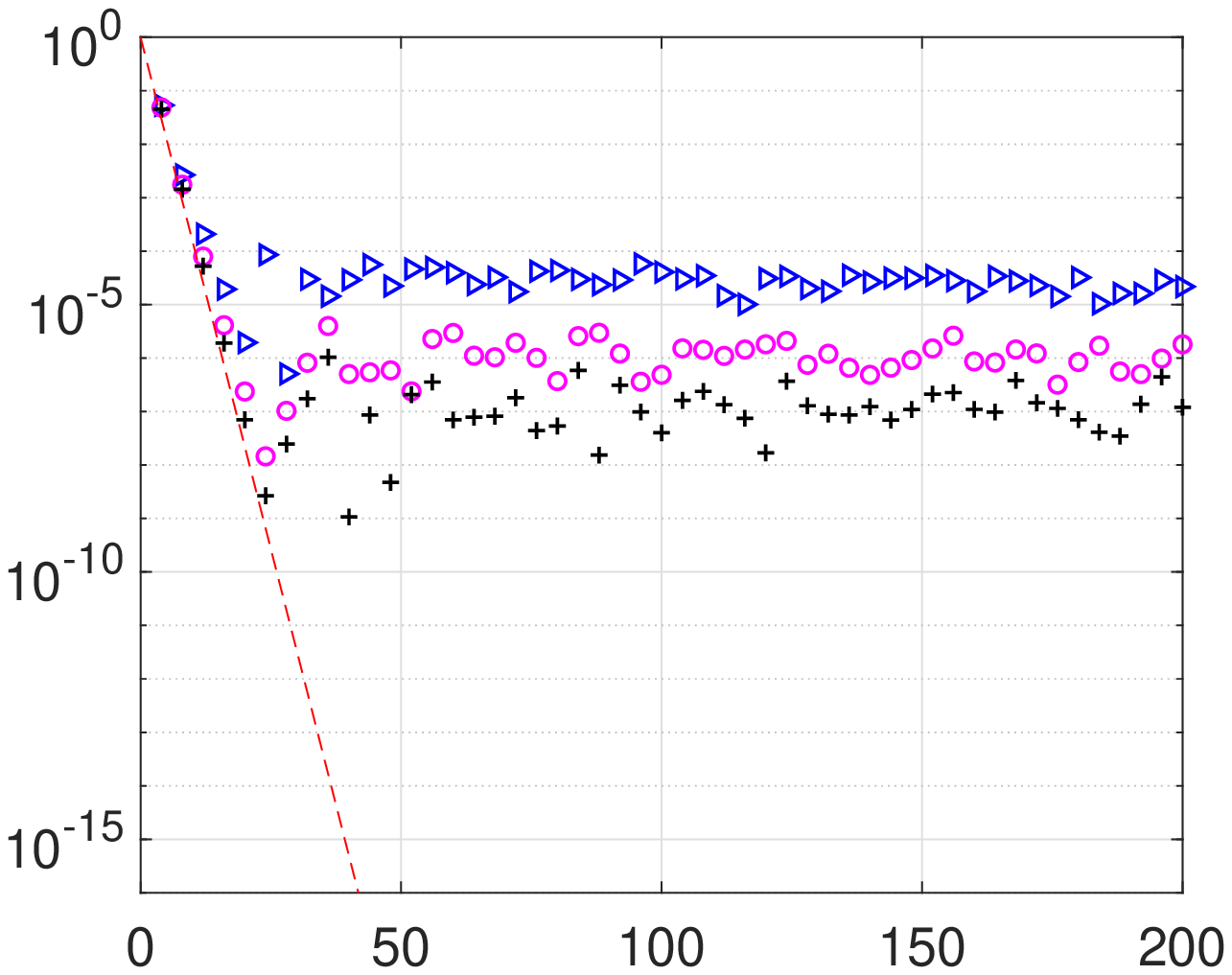}}
{\includegraphics[width=5cm]{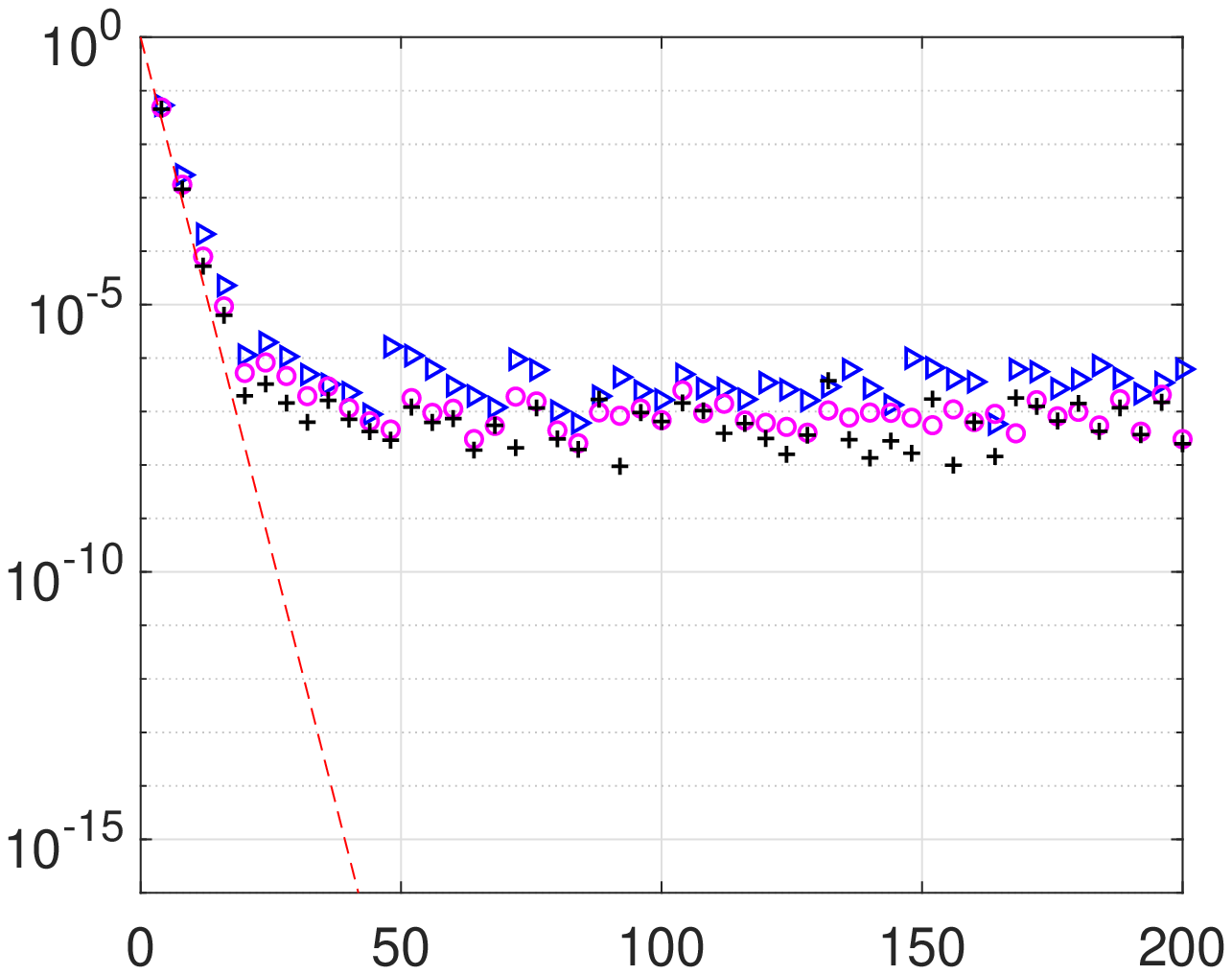}}
{\includegraphics[width=5cm]{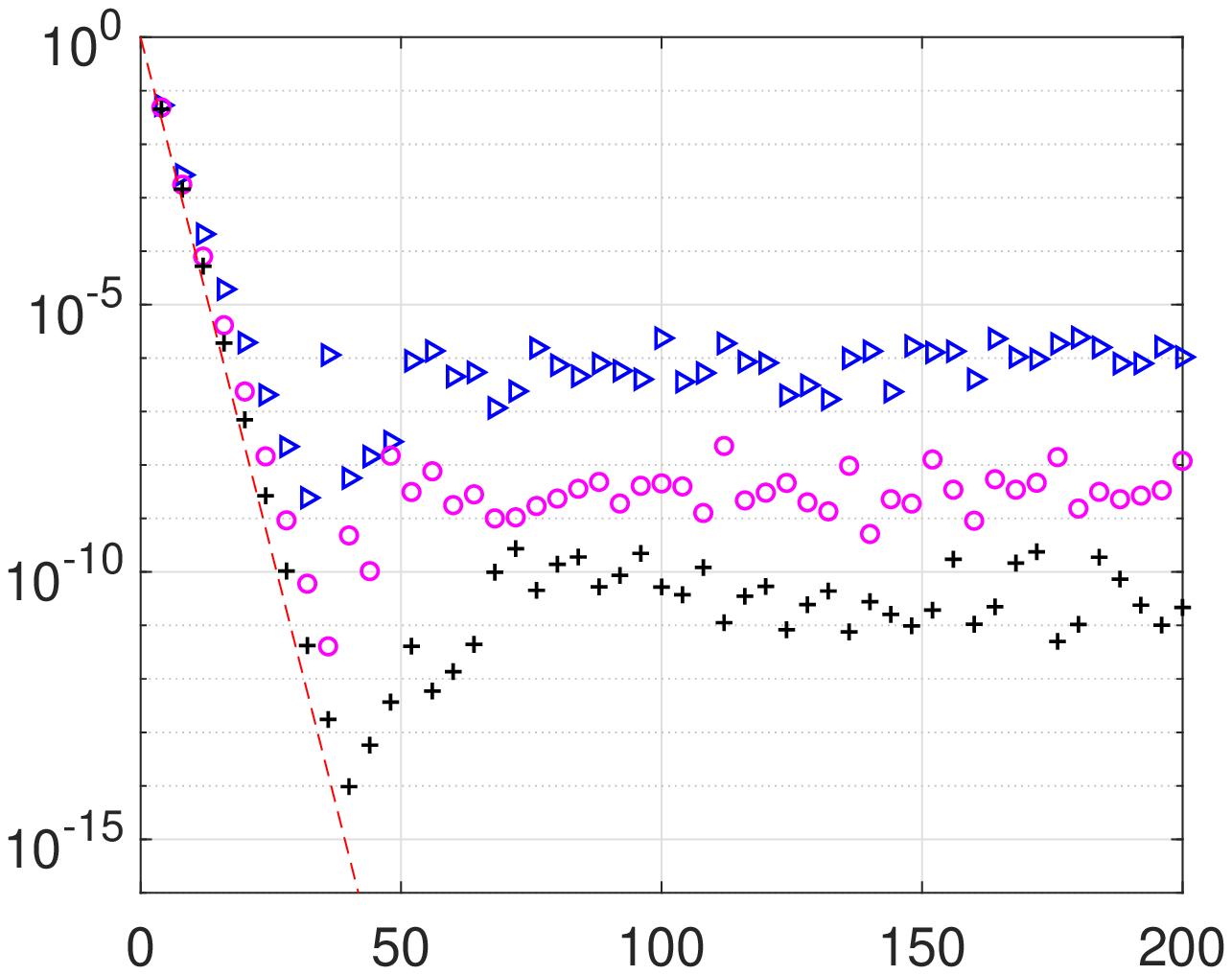}}
{\includegraphics[width=5cm]{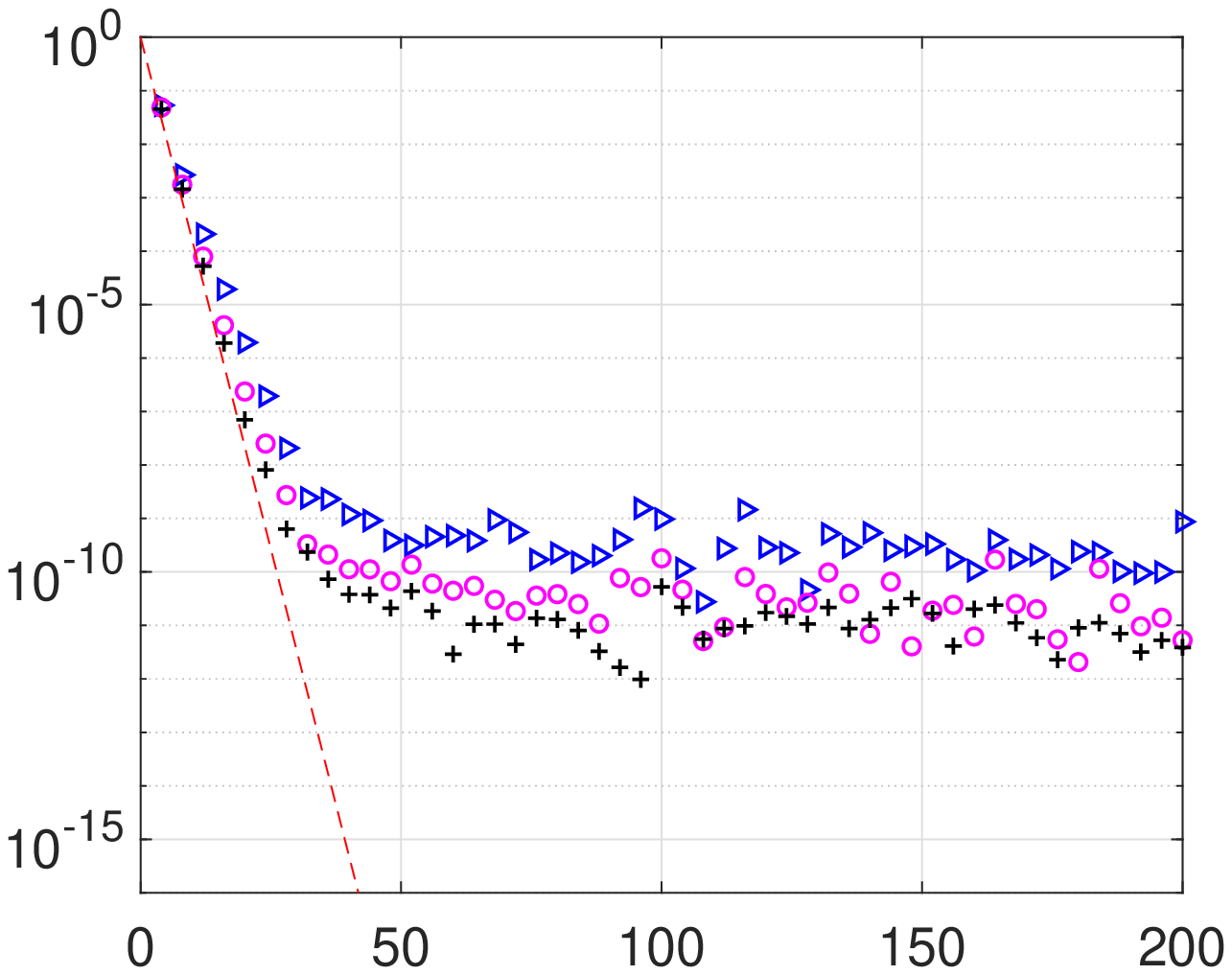}}
{\includegraphics[width=5cm]{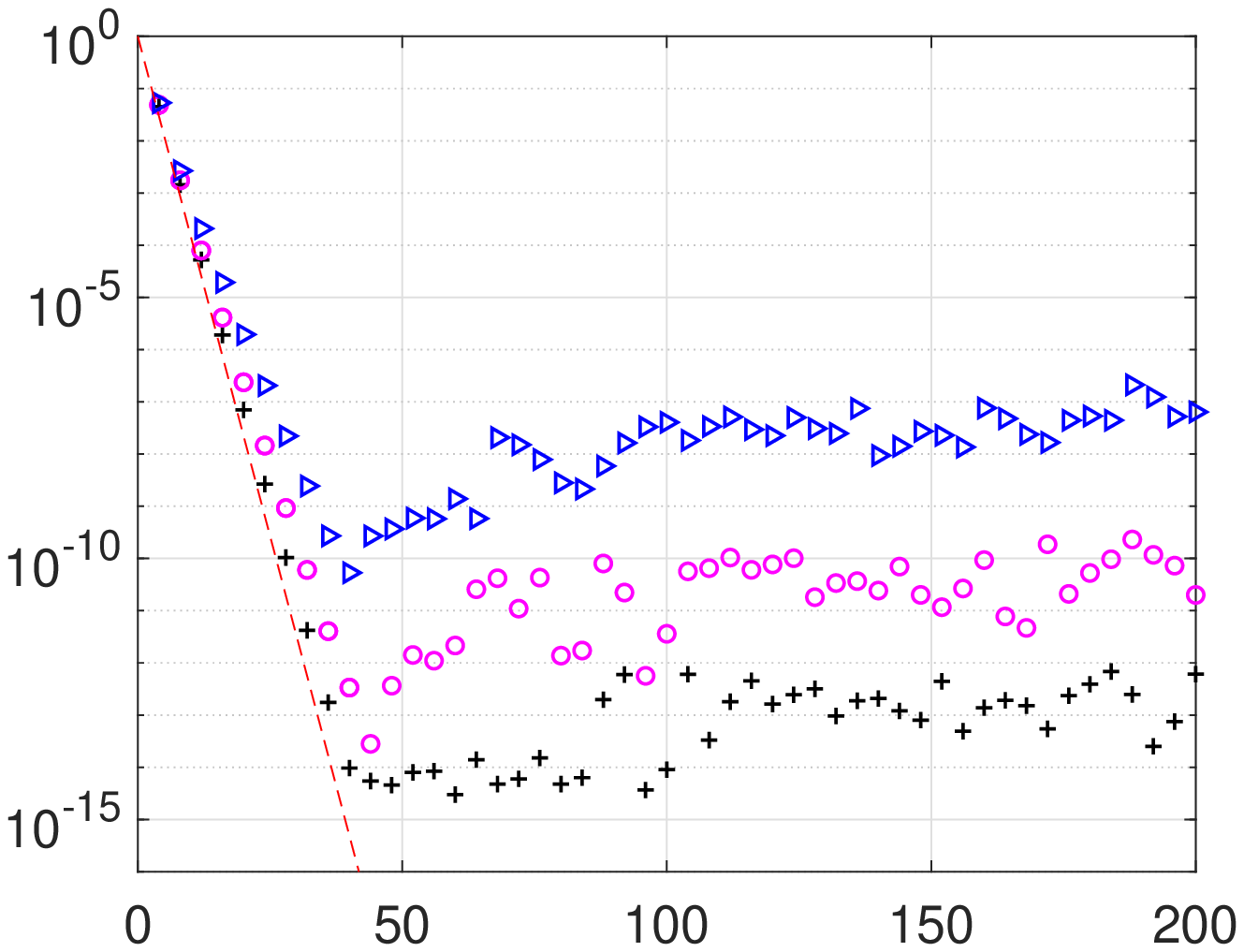}}
{\includegraphics[width=5cm]{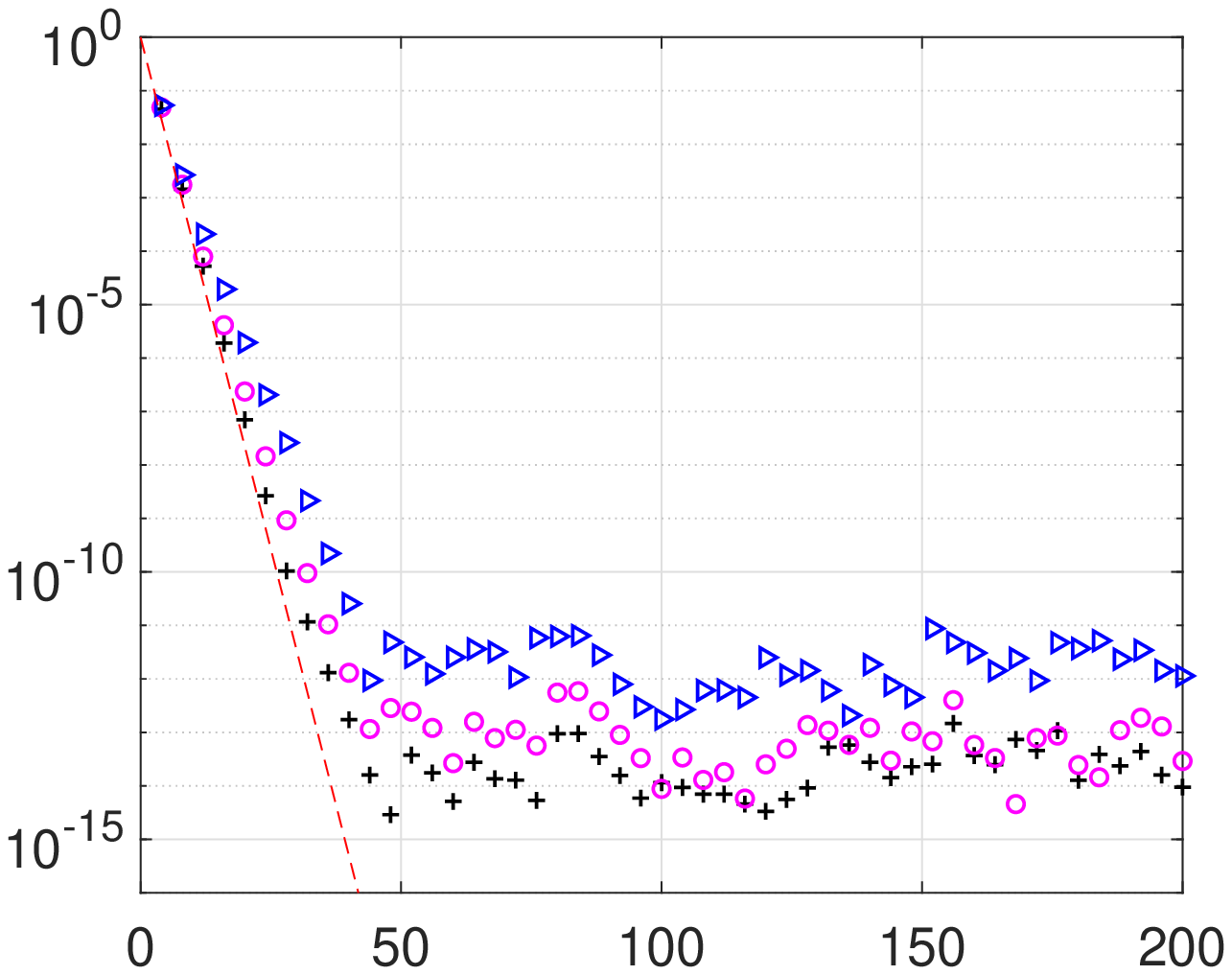}}
\caption{Uniform approximation error versus $n$ for approximation function $f_1$ via $\mathcal{Q}^{\varepsilon,\gamma}_{m,n,\alpha,\beta}$ using various different values of $\gamma,\varepsilon,\eta$. The values of $\gamma$ are $\gamma = 1.2$ (left) and $\gamma = 1.8$ (right). The values of $\varepsilon$ are $\varepsilon=10^{-6}$ (top), $\varepsilon=10^{-10}$ (middle) and $\varepsilon=10^{-14}$ (bottom). The values of $\eta$ are
$\eta=1.25$ (blue triangles), $\eta=2$ (magenta circle) and $\eta=4$ (black plus). The red line indicates $\theta_1^{-n}$.\label{two}}
\end{figure}

\begin{figure}[th]
\centering
\subfigure[$f_1$]
{\psfig{file=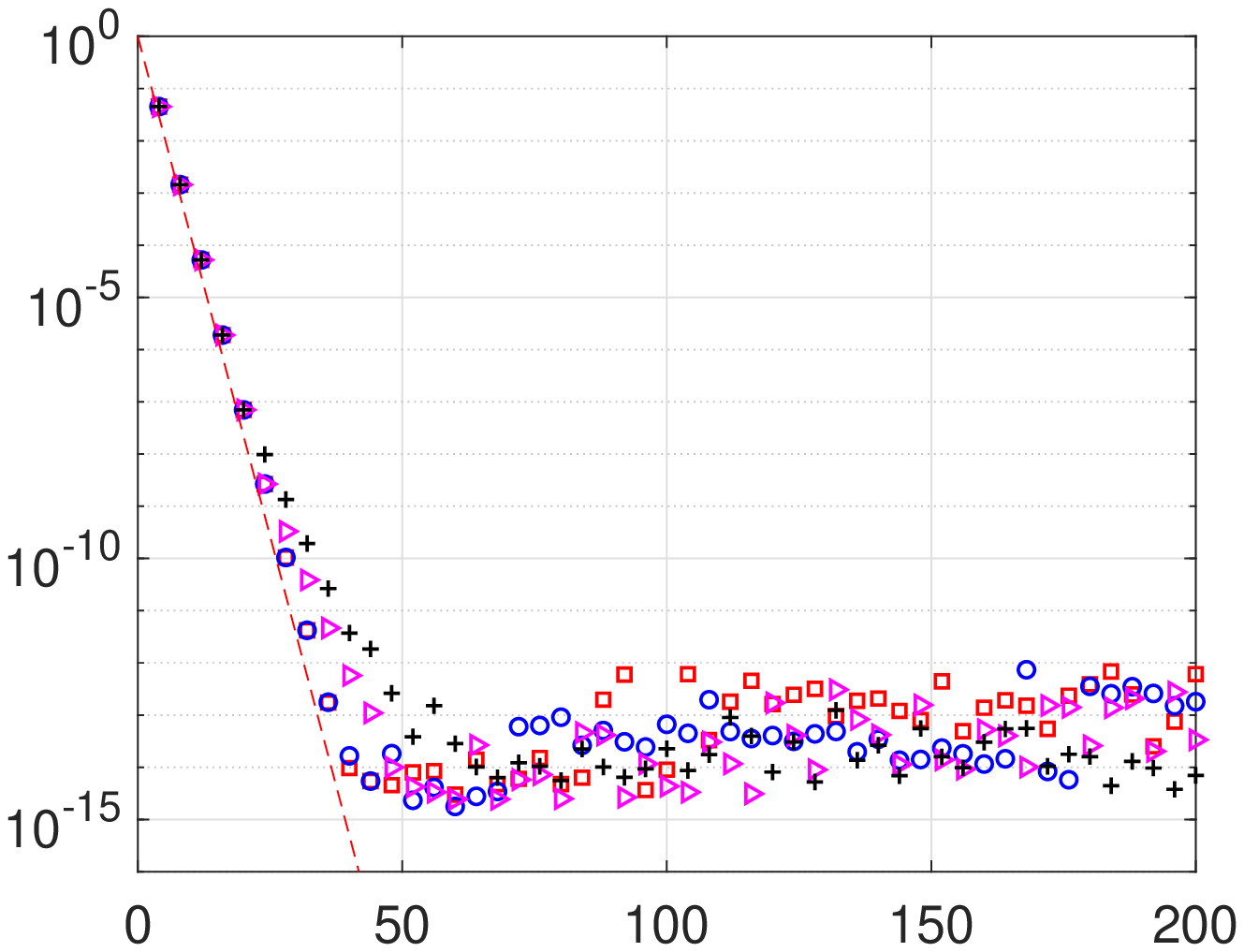,width=5cm}}
\quad
\subfigure[$f_2$]
{\psfig{file=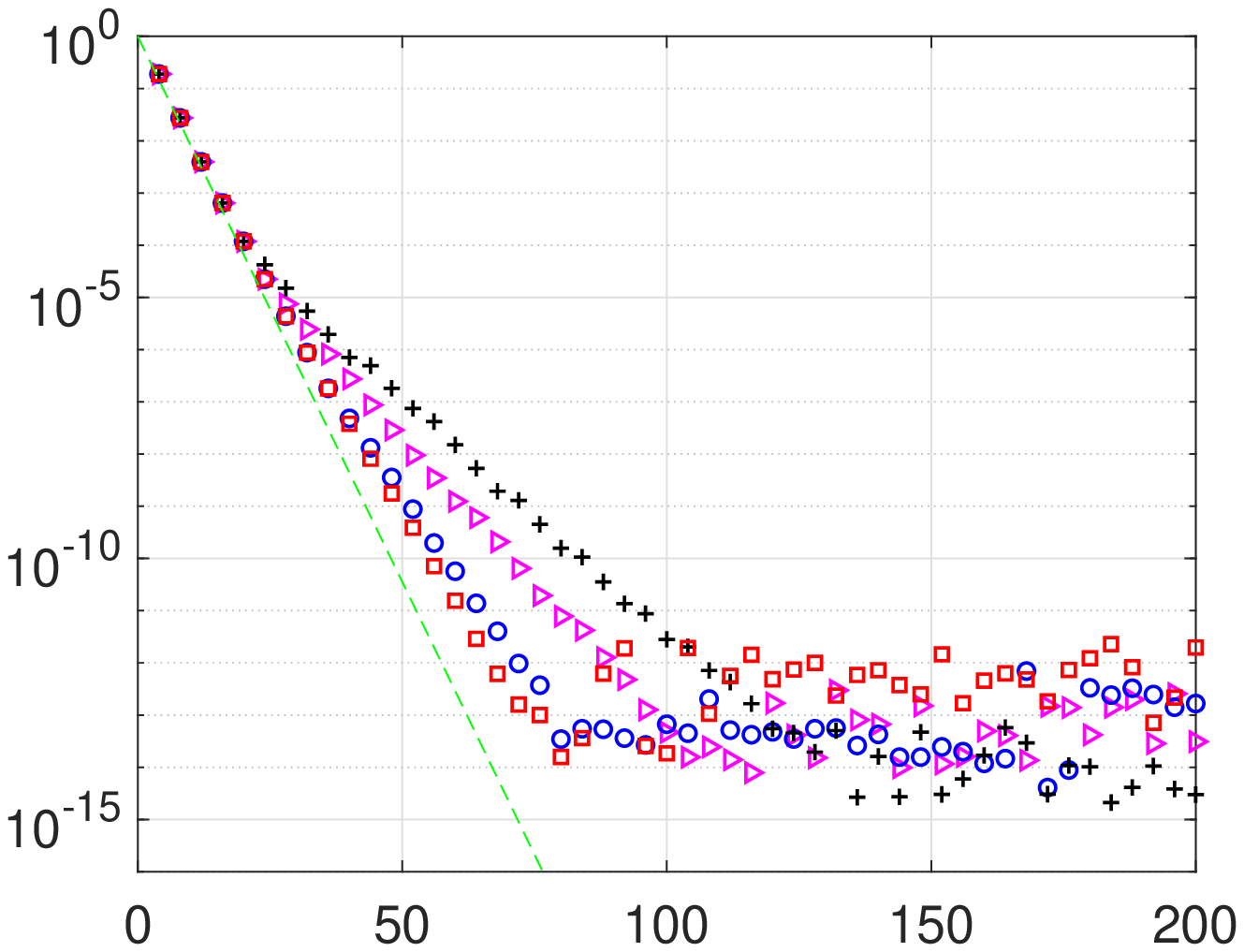,width=5cm}}
\quad
\subfigure[$f_3$]
{\psfig{file=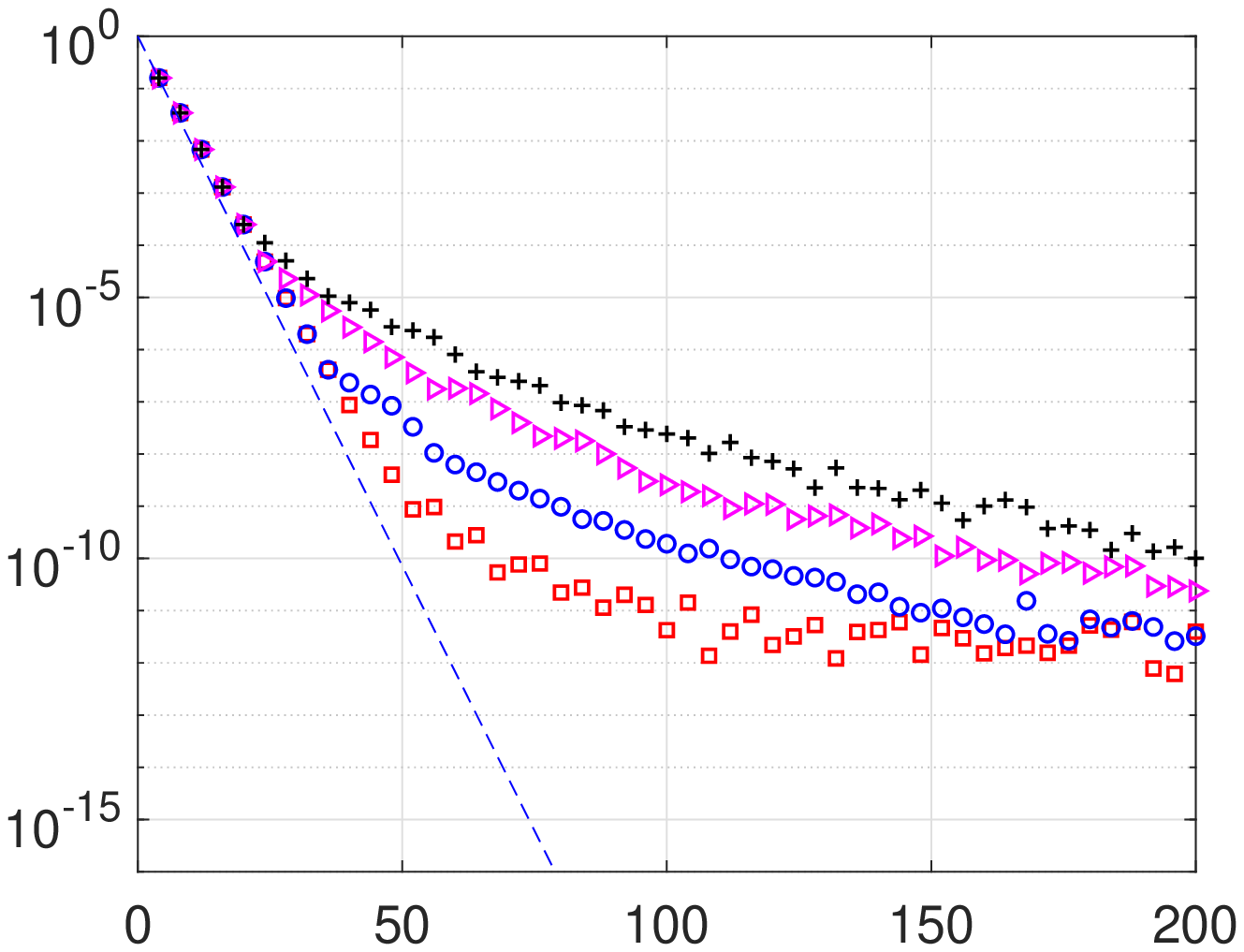,width=5cm}}
\quad
\subfigure[$f_4$]
{\psfig{file=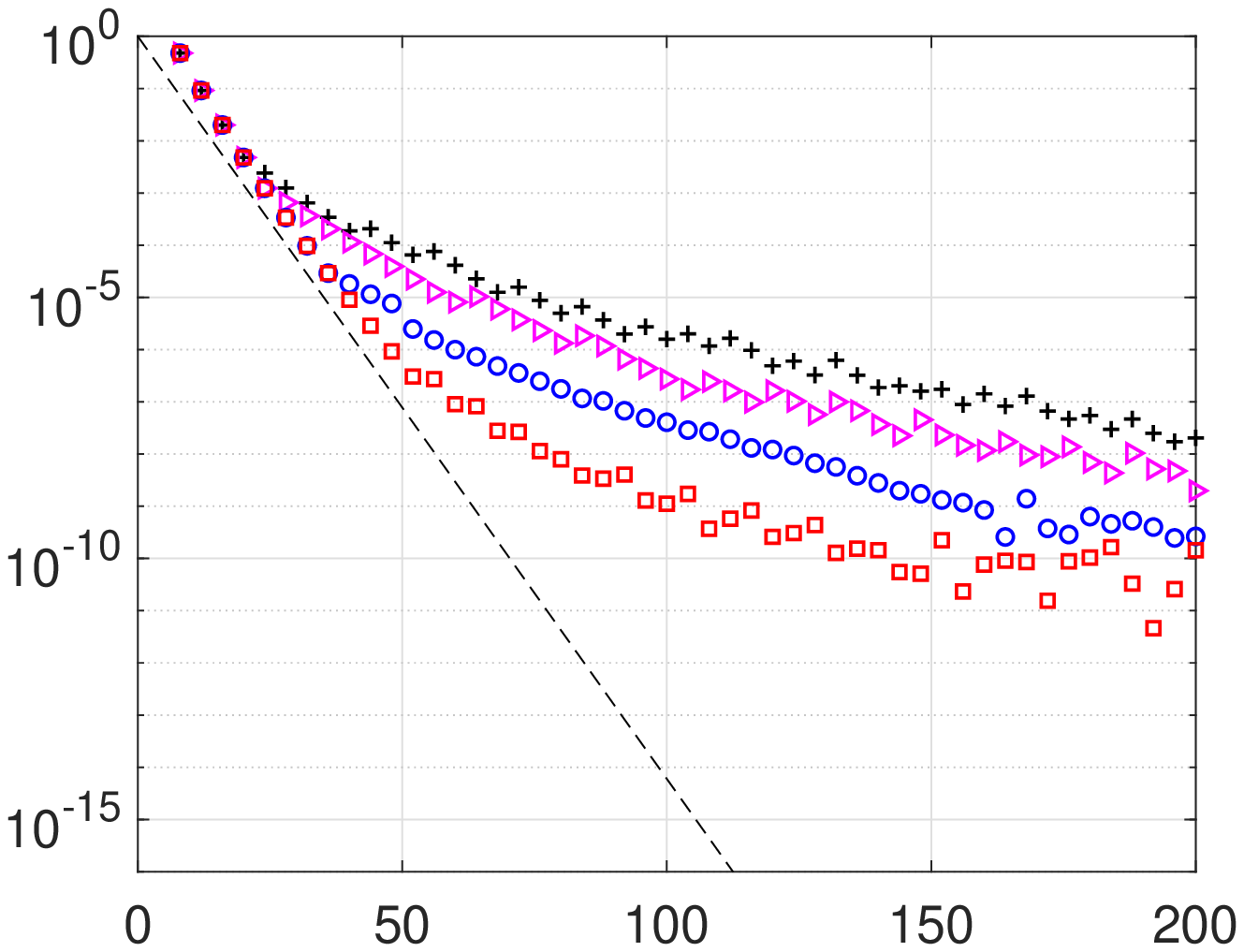,width=5cm}}
\vspace*{8pt}
\caption{Uniform Jacobi frame approximation error versus $n$ for functions $f_1$, $f_2$, $f_3$ and $f_4$ via $\mathcal{Q}^{\varepsilon,\gamma}_{m,n,\alpha,\beta}$ using various values of $\gamma$, i.e., $\gamma=1.2$ (red square), $\gamma=1.5$ (blue circle), $\gamma=2$ (magenta triangle) and $\gamma=2.5$ (black plus).
The red line indicates $\theta_1^{-n}$, the green line indicates $\theta_2^{-n}$, the blue line indicates $\theta_3^{-n}$ and the black line indicates $\theta_4^{-n}$.
\label{three}}
\end{figure}

In fact, Adcock and Shadrin concluded that the limiting accuracy \eqref{apperror} gets better with increasing the oversampling parameter $\eta$ for Legendre frame approximation. They also noticed that increasing $\eta$ makes less difference to the accuracy when $\varepsilon$ is quite smaller. These two conclusions still hold for the Jacobian frame approximation. Hence, we fix $\varepsilon=10^{-14}$ and take $\eta=4$ in following sections, and we focus on the influence of parameters $\gamma,\alpha,\beta$.

\subsection{The influences of parameter $\gamma$}
In Figure~\ref{three}, we plot the uniform Jacobi frame approximation error versus $n$ of functions $f_1$, $f_2 = 1/(1+4x^2)$, $f_3= 1/(10-9x)$ and $f_4 =25\sqrt{9x^2-10}$ for different values of the extended domain parameter $\gamma$. These functions are analytic in Bernstein ellipses $B(E_\theta)$ with parameters $\theta_1$, $\theta_2=(1+\sqrt{5})/2$, $\theta_3 = (10+\sqrt{19})/9$ and $\theta_4 = \sqrt{10/9}+1/3$ respectively.

We do witness exponential decrease of the error down to some fixed limiting accuracy for function $f_1$, which is analytic in Bernstein ellipse $B(E_{\theta_1})$ that is large enough to contain the extended interval $[-\gamma,\gamma]$. Moreover, the error is larger when $\gamma$ is smaller, and the error is smaller when $\gamma$ is larger. On the contrary, for functions $f_2$ $f_3$ and $f_4$ that are not analytic in complex regions containing the extended interval $[-\gamma,\gamma]$, we see that the error still decays with exponential rate but only down to a larger tolerance, and the approximation accuracy with smaller value $\gamma = 1.2$ is better rather than a larger value $\gamma=2.5$.

We then investigate the Jacobi frame approximation error with various $\gamma$ of differentiable functions $f_5 = |x|$, $f_6 = |\sin x|^3$, $f_7 = |x-1/2|^5$ and $f_8 = |x-1/4|^{3/2}$ in Figure~\ref{four}. Here we use much larger values of $\gamma$. When approximating differentiable functions with a larger $\gamma$, we obviously observe that the approximation accuracy will lost. For instance, for the $3$-times differentiable function $f_5$, the accuracy when $\gamma = 15$ is one order of magnitude worse than the accuracy when $\gamma=2$. Moreover, by comparing the error of each approximated function in Figure~\ref{three} and Figure~\ref{four}, we find that the higher the smoothness of the approximated function, the more significant the loss of approximation accuracy.

\begin{figure}[bh]
\centering
\subfigure[$f_5$]
{\psfig{file=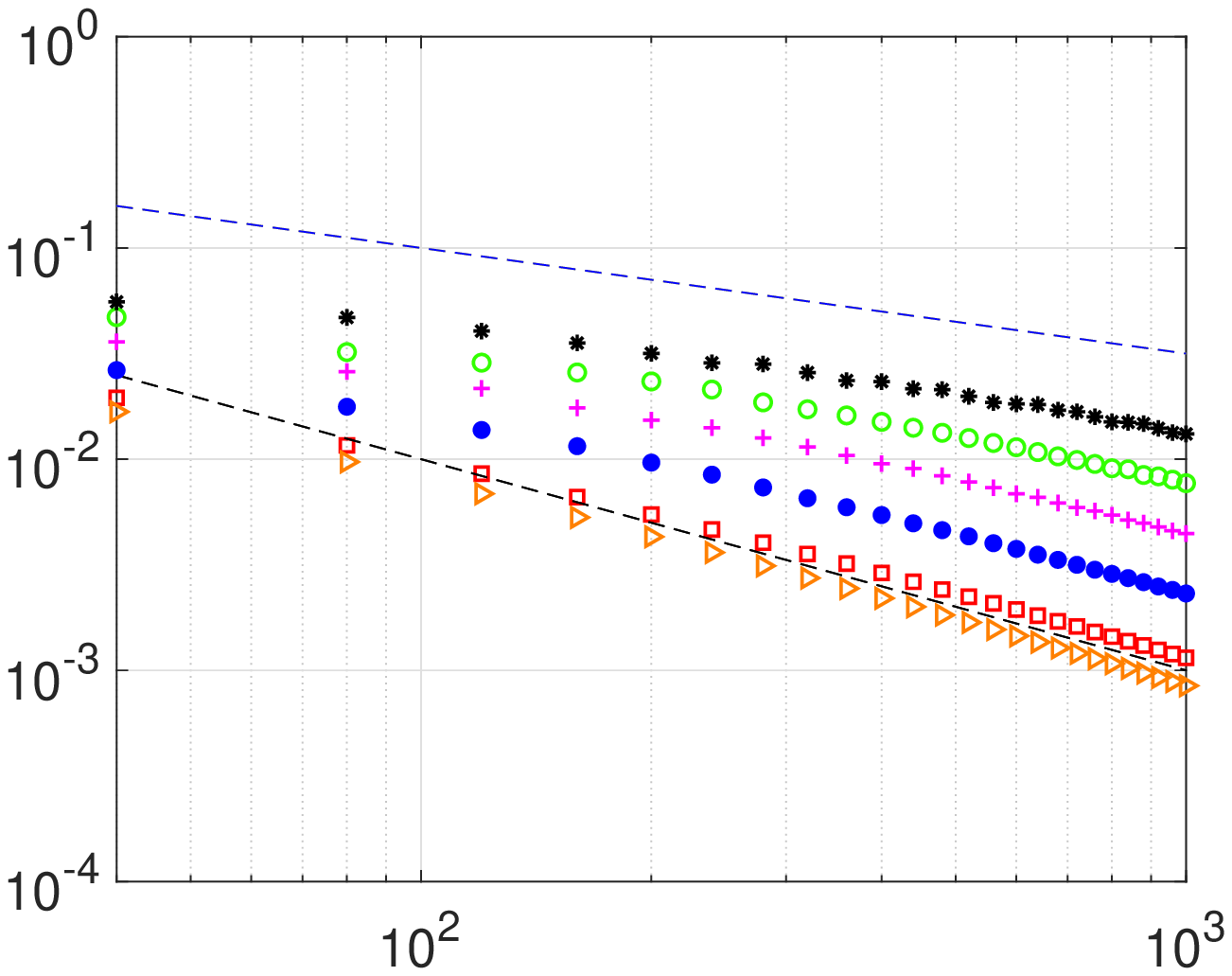,width=5cm}}
\quad
\subfigure[$f_6$]
{\psfig{file=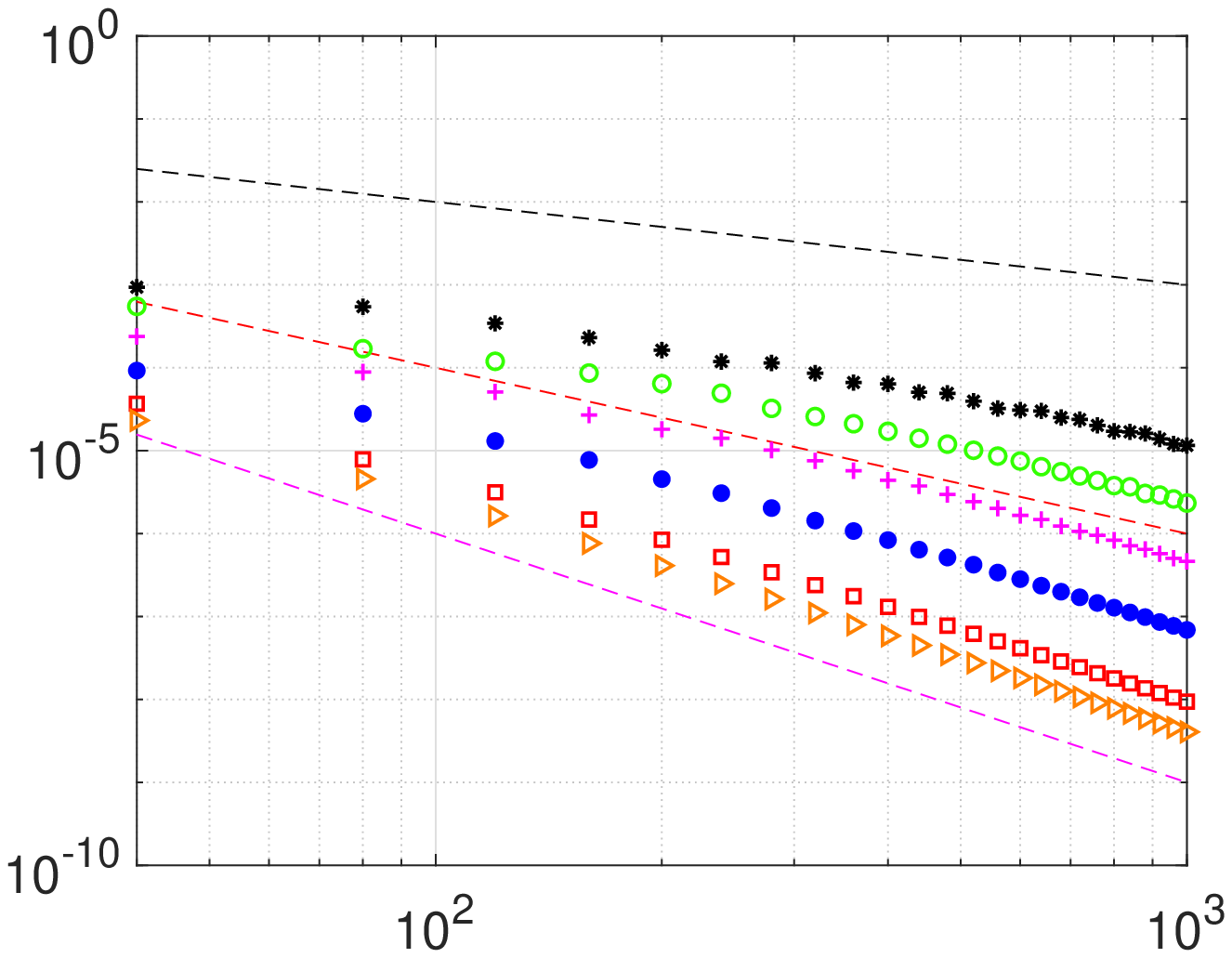,width=5cm}}
\quad
\subfigure[$f_7$]
{\psfig{file=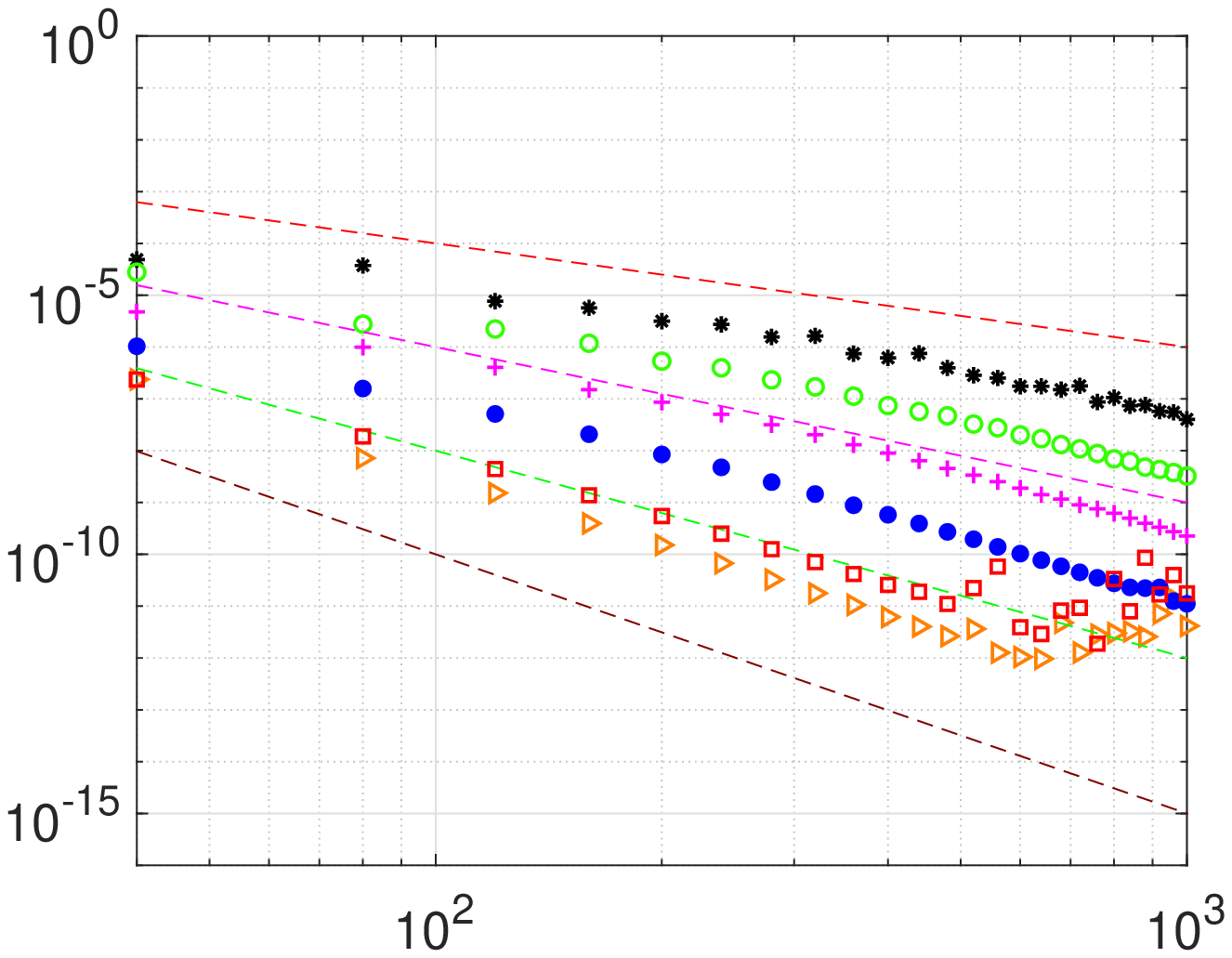,width=5cm}}
\quad
\subfigure[$f_8$]
{\psfig{file=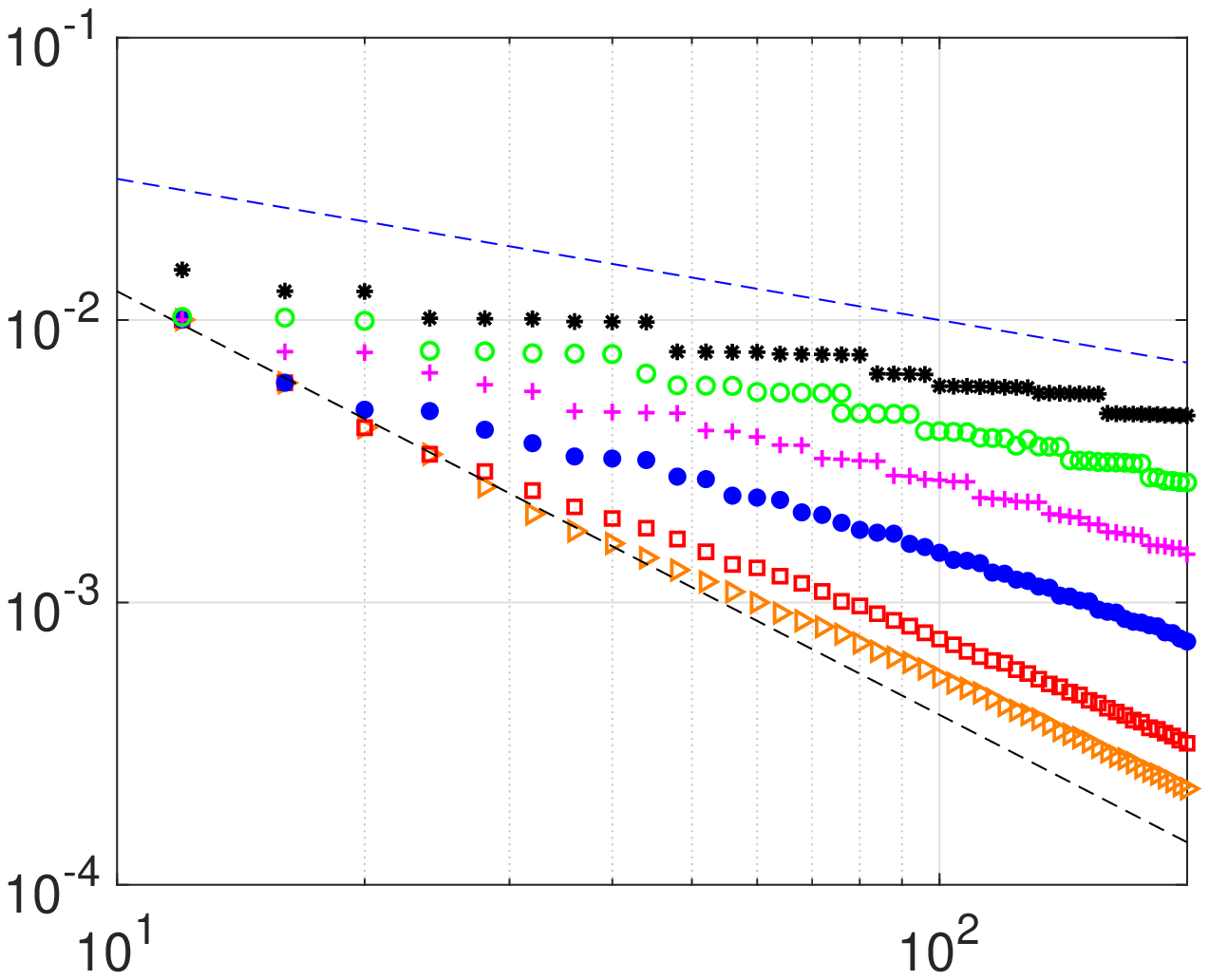,width=5cm}}
\vspace*{8pt}
\caption{Uniform Jacobi frame approximation error versus $n$ for functions $f_5$, $f_6$, $f_7$, $f_8$ via $\mathcal{Q}^{\varepsilon,\gamma}_{m,n,\alpha,\beta}$ using various values of $\gamma$, i.e., i.e., $\gamma=1.5$ (orange triangle), $\gamma=2$ (red square), $\gamma=4$ (blue dot), $\gamma=8$ (magenta plus), $\gamma=15$ (green circle), $\gamma=30$ (black star).
The blue line indicates the quantity $n^{-1/2}$, the black line in subfigure $(a)$ indicates the quantity $n^{-1}$, the red line indicates the quantity $n^{-2}$, the magenta line indicates the quantity $n^{-3}$, the green line indicates the quantity $n^{-4}$, the crimson line indicates the quantity $n^{-5}$, the black line in subfigure $(d)$ indicates the quantity $n^{-3/2}$.
\label{four}}
\end{figure}

\begin{figure}[th]
\centering
\subfigure[$f_1$]
{\psfig{file=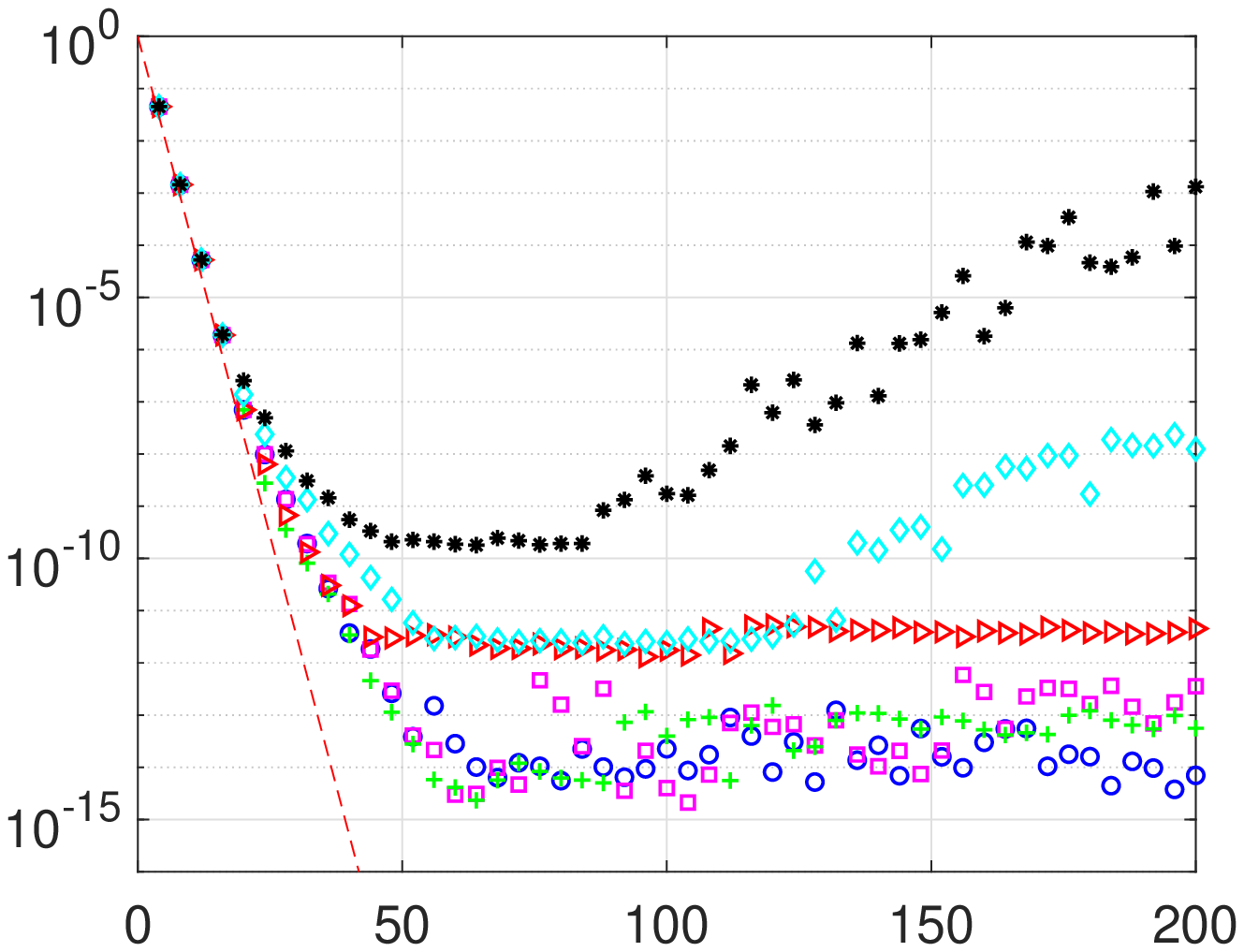,width=5cm}}
\quad
\subfigure[$f_2$]
{\psfig{file=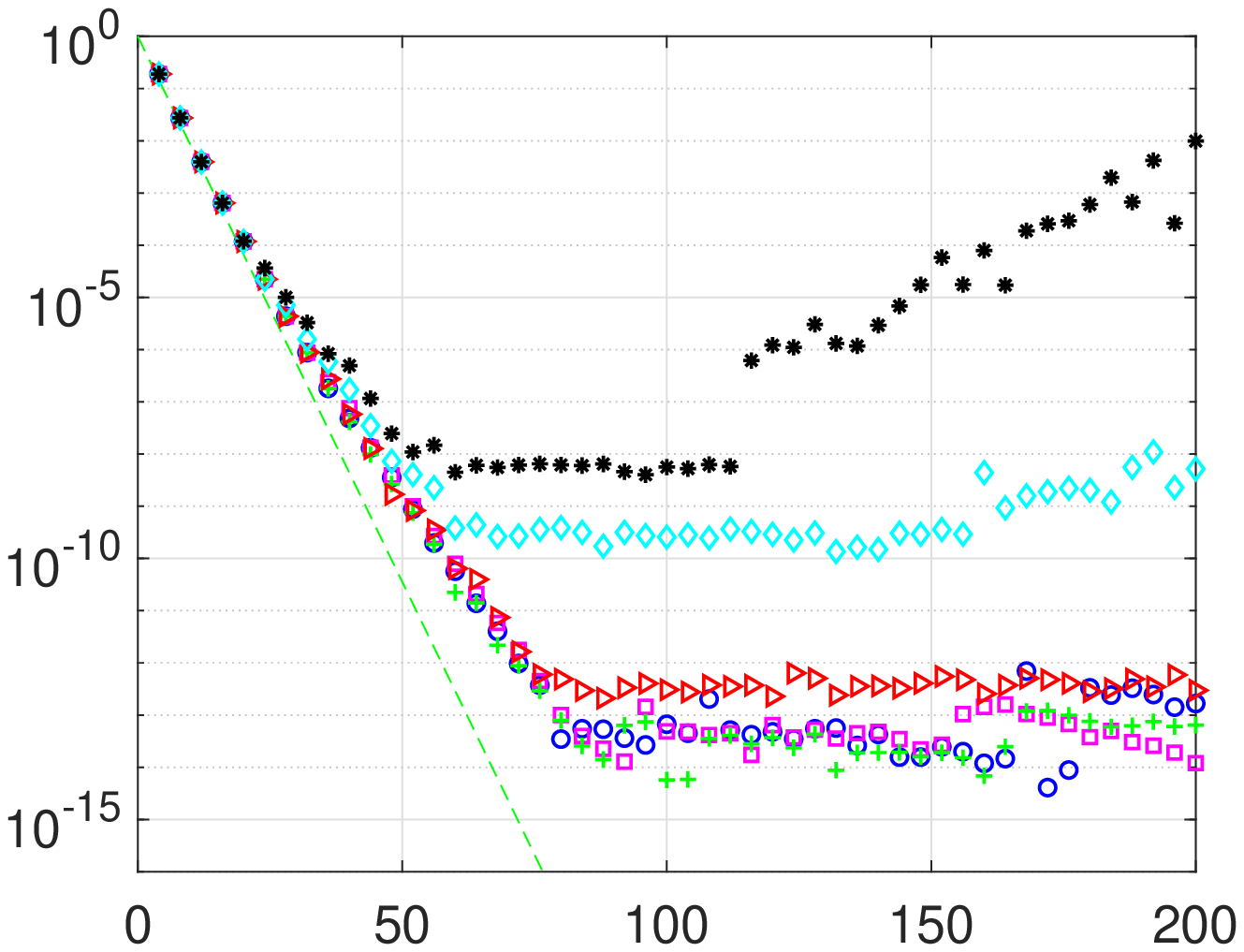,width=5cm}}
\quad
\subfigure[$f_3$]
{\psfig{file=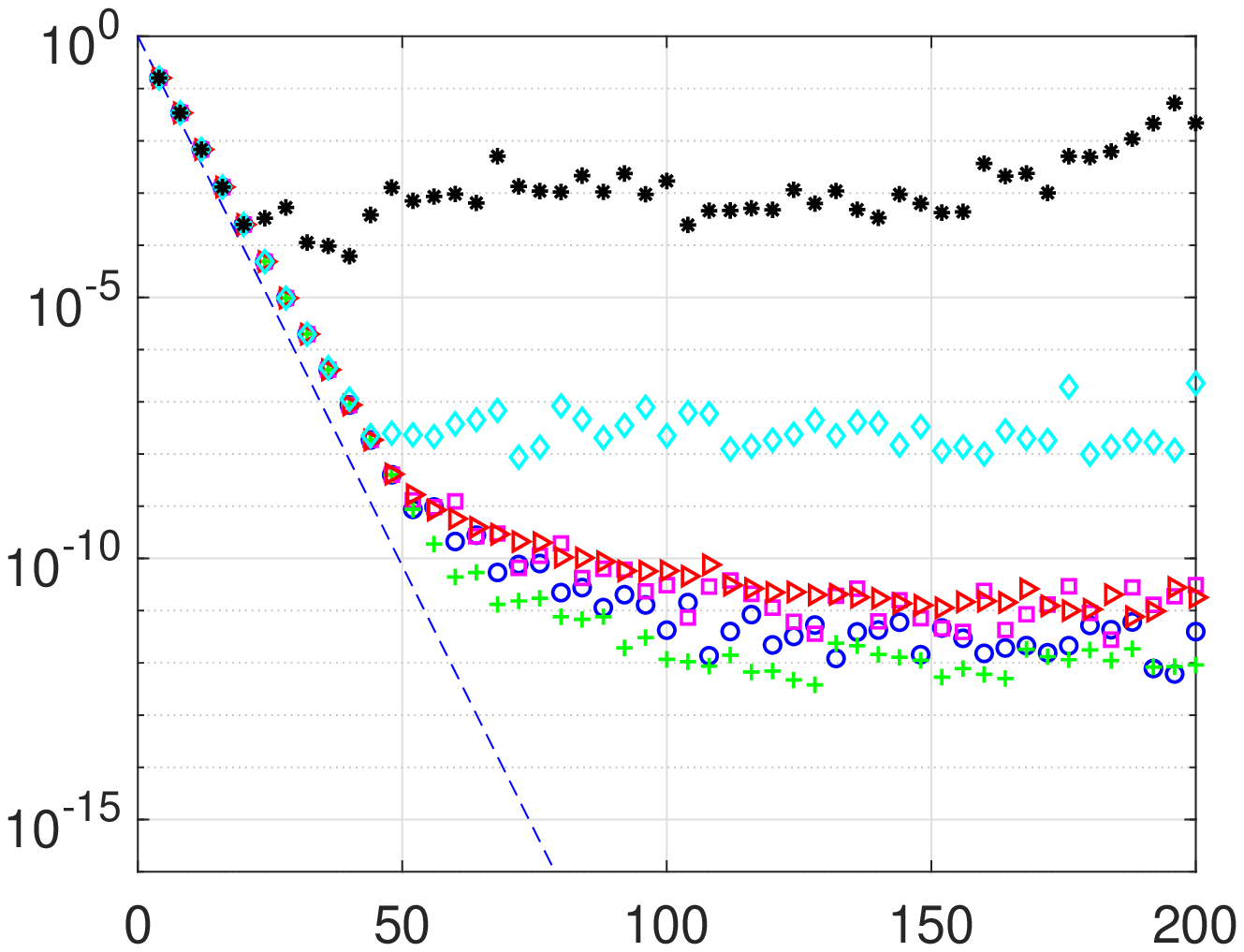,width=5cm}}
\quad
\subfigure[$f_4$]
{\psfig{file=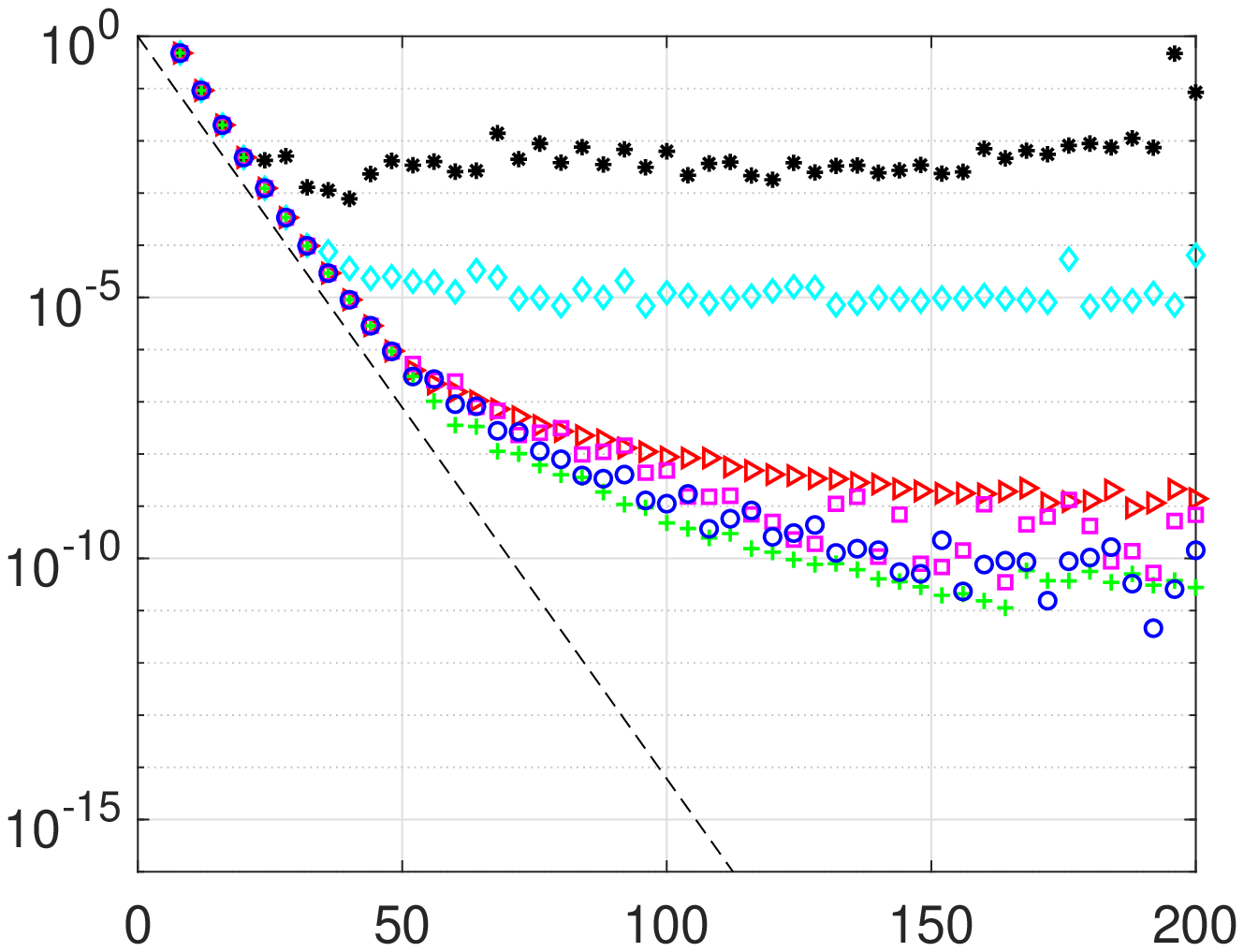,width=5cm}}
\vspace*{8pt}
\caption{Uniform Jacobi frame approximation error versus $n$ for functions $f_1$, $f_2$, $f_3$, $f_4$ via $\mathcal{Q}^{\varepsilon,\gamma}_{m,n,\alpha,\beta}$ using various values of $(\alpha,\beta)$, i.e., $(\alpha,\beta)=(1/3,1/2)$ (blue circle), $(\alpha,\beta)=(-1/3,-2/3)$ (magenta square), $(\alpha,\beta)=(2,5/2)$ (green plus), $(\alpha,\beta)=(5,10)$ (red triangle), $(\alpha,\beta)=(15,1)$ (cyan diamond) and $(\alpha,\beta)=(0,20)$ (black star). The red line indicates $\theta_1^{-n}$, the green line indicates $\theta_2^{-n}$, the blue line indicates $\theta_3^{-n}$ and the black line indicates $\theta_4^{-n}$. \label{five}}
\end{figure}

\subsection{The influences of parameters $\alpha$ and $\beta$}
Then we investigate the influences of parameters $\alpha$ and $\beta$ on the error decay. For functions $f_1$, $f_2$, $f_3$ and $f_4$, we fix the extended domain parameter $\gamma = 2.5$, $1.5$, $1.2$, $1.2$ respectively. As shown in Figure~\ref{five}, it can be found that when $\max\{\alpha,\beta\}>10$, parameters $\alpha$ and $\beta$ begin to affect the decay of error. The accuracy of the approximation will become worse or even diverge. This is consistent with \eqref{apperror}. Meanwhile, we observe the similar phenomenons for differentiable functions $f_5$, $f_6$, $f_7$ and $f_9 = |x+1/2|^7$ with other six sets of parameters $(\alpha,\beta)$ and fixed $\gamma=2$ in Figure~\ref{six}.

\clearpage
\begin{figure}[th]
\centering
\subfigure[$(\alpha, \beta)= (1/3,-1/2)$]
{\includegraphics[width=5cm]{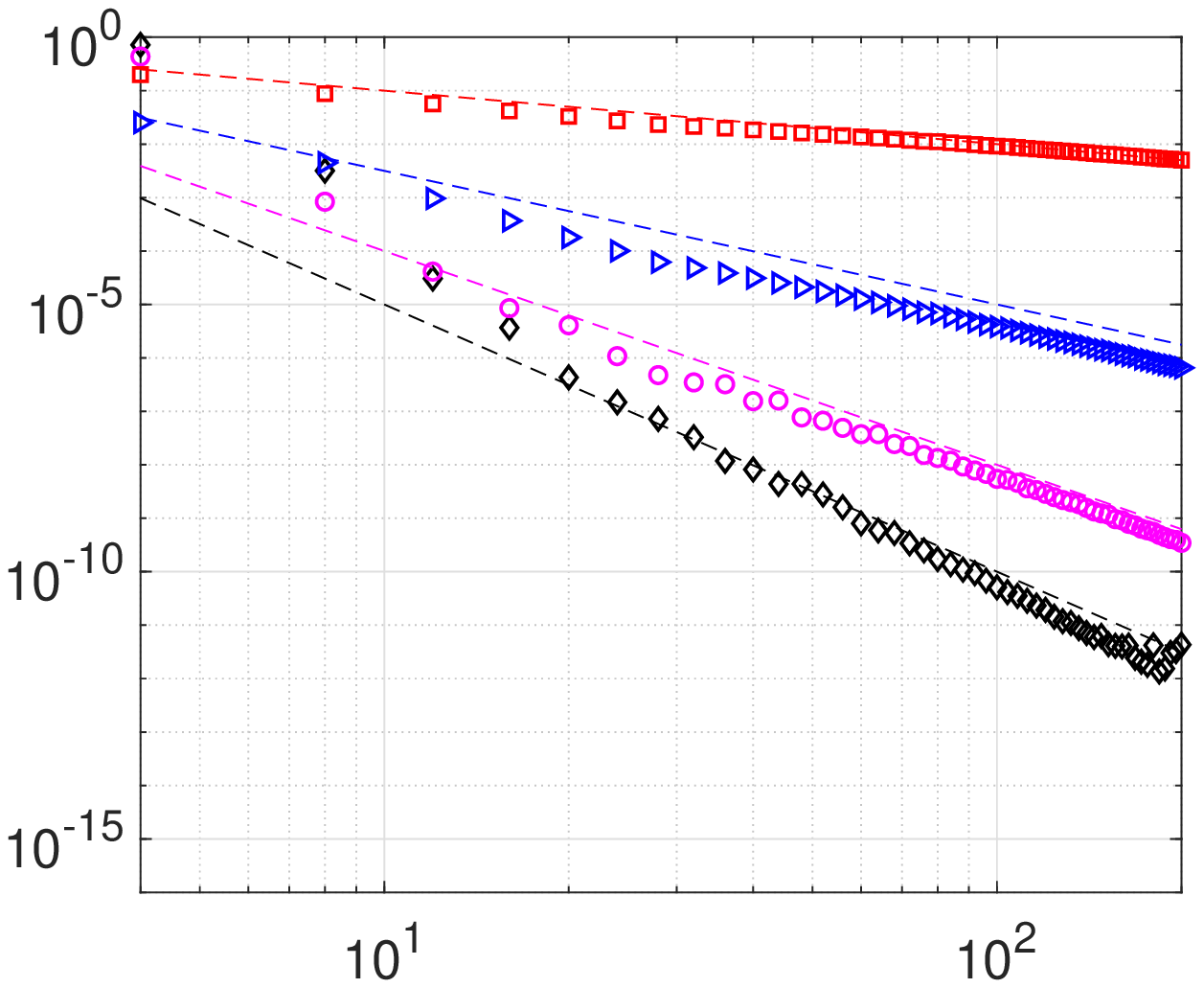}}
\subfigure[$(\alpha, \beta)=(-1/3,1)$]
{\includegraphics[width=5cm]{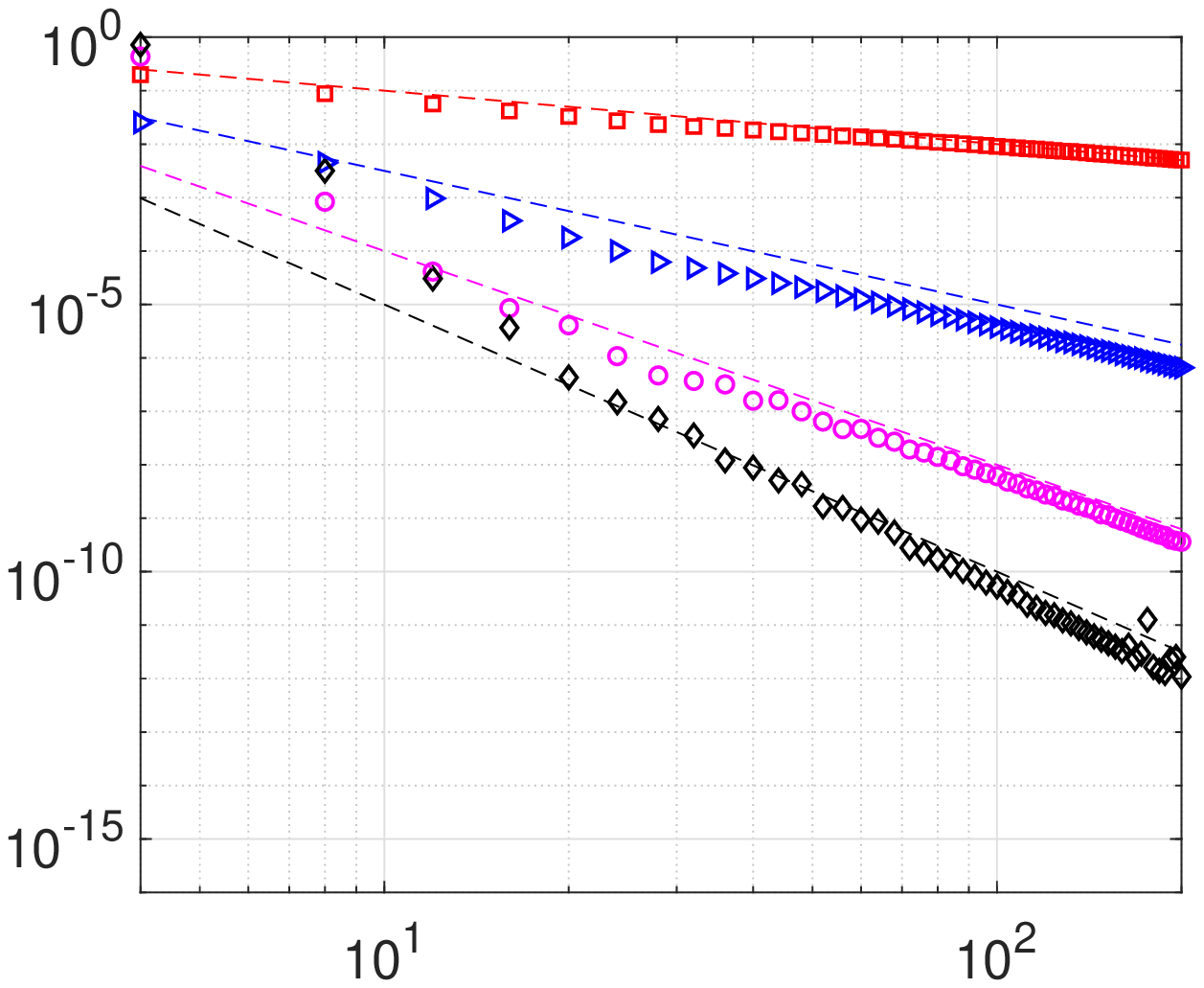}}
\subfigure[$(\alpha, \beta)= (1,2)$]
{\includegraphics[width=5cm]{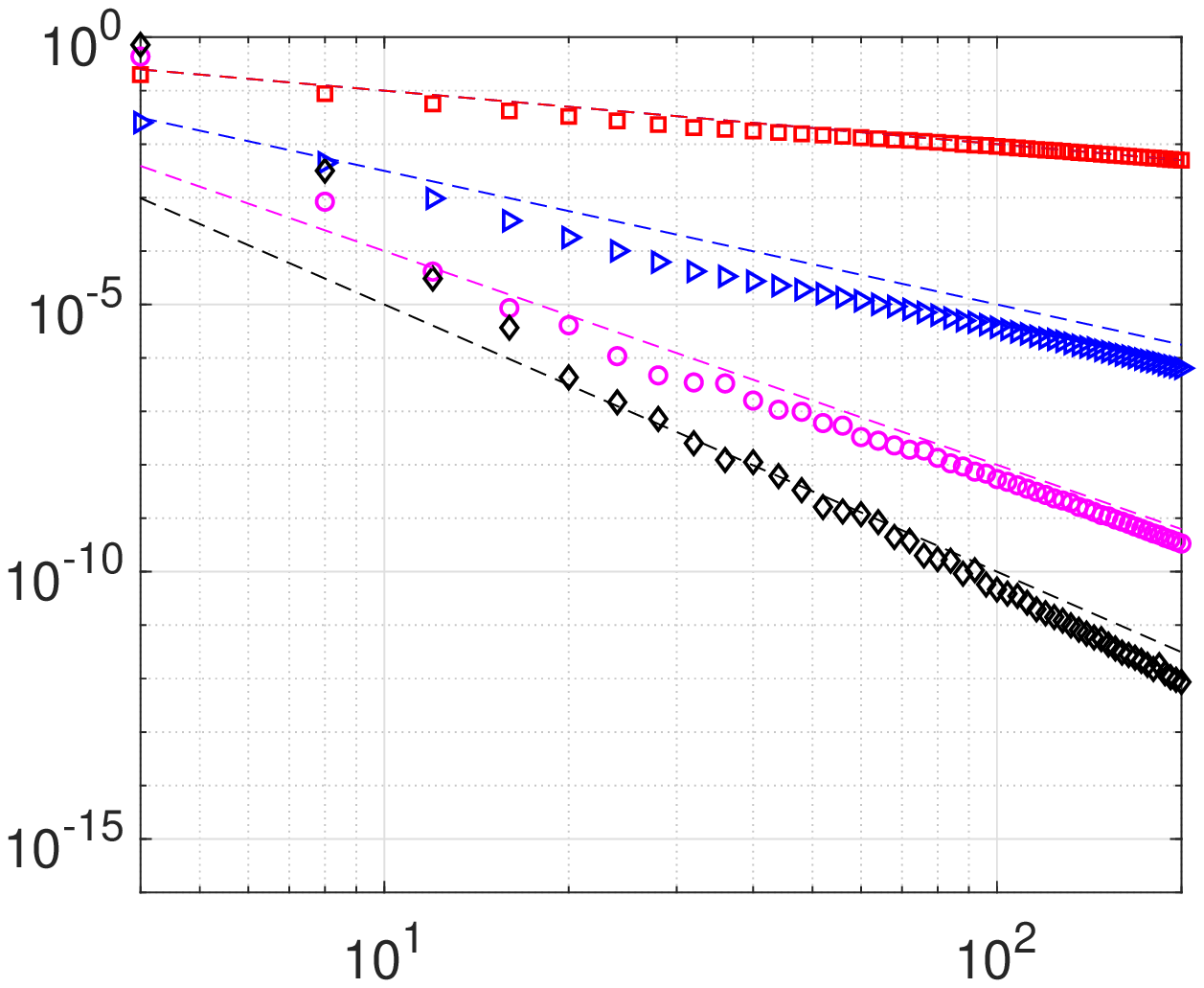}}
\subfigure[$(\alpha, \beta)= (10,2)$]
{\includegraphics[width=5cm]{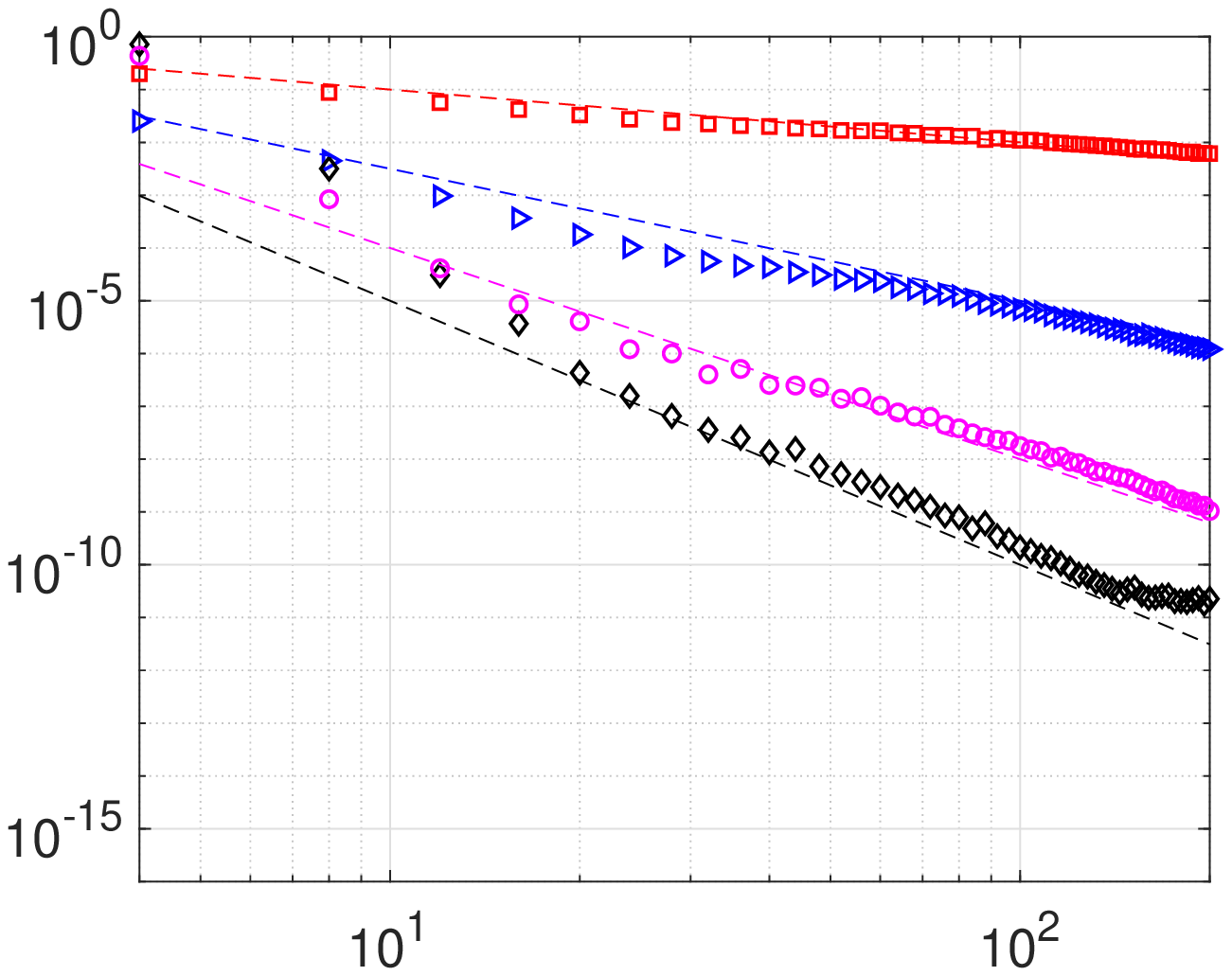}}
\subfigure[$(\alpha, \beta)= (1,15)$]
{\includegraphics[width=5cm]{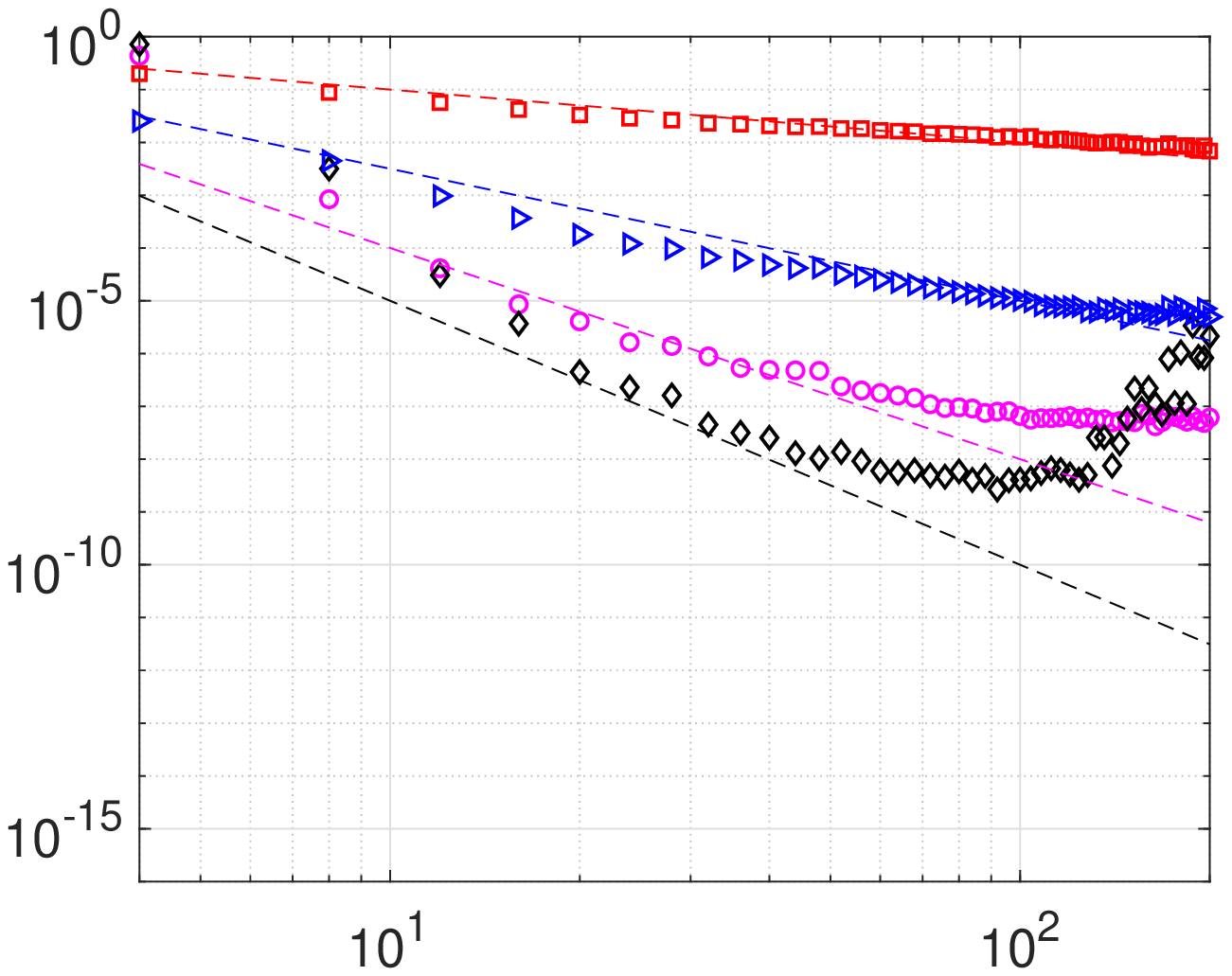}}
\subfigure[$(\alpha, \beta)= (20,0)$]
{\includegraphics[width=5cm]{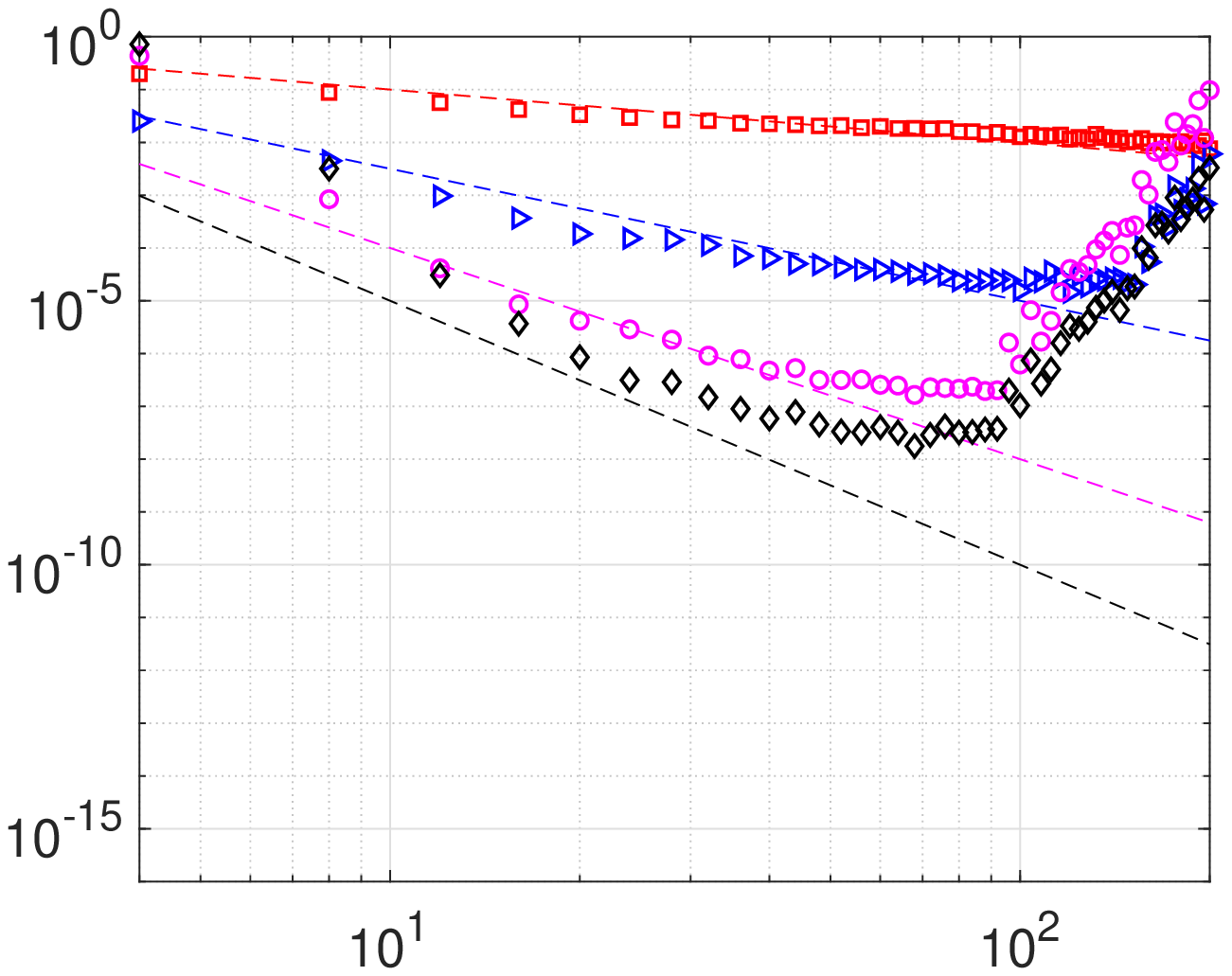}}
\vspace*{8pt}
\caption{Uniform Jacobi frame approximation error versus $n$ for functions $f_5$ (red square), $f_6$ (blue triangle), $f_7$ (magenta circle) and $f_9$ (black diamond) via $\mathcal{Q}^{\varepsilon,\gamma}_{m,n,\alpha,\beta}$ using various values of $(\alpha,\beta)$. The red line indicates the quantity $n^{-1}$, the blue line indicates the quantity $n^{-5/2}$, the magenta line indicates the quantity $n^{-4}$, the black line indicates the quantity $n^{-5}$. \label{six}}
\end{figure}

\end{document}